\documentclass[12pt,a4paper]{amsart}

\usepackage[T1]{fontenc}
\usepackage{amssymb,amsmath,latexsym,amsthm,paralist,a4,mathrsfs,dsfont}
\usepackage{stmaryrd}
\usepackage{tikz-cd}
\usetikzlibrary{decorations.markings}
\usepackage{thmtools}

\makeatletter
\tikzcdset{
  open/.code     = {\tikzcdset{hook, circled};},
  closed/.code   = {\tikzcdset{hook, slashed};},
  open'/.code    = {\tikzcdset{hook', circled};},
  closed'/.code  = {\tikzcdset{hook', slashed};},
  circled/.code  = {\tikzcdset{markwith = {\draw (0,0) circle (.375ex);}};},
  slashed/.code  = {\tikzcdset{markwith = {\draw[-] (-.4ex,-.4ex) -- (.4ex,.4ex);}};},
  markwith/.code ={
    \pgfutil@ifundefined%
    {tikz@library@decorations.markings@loaded}%
    {\pgfutil@packageerror{tikz-cd}{You need to say %
      \string\usetikzlibrary{decorations.markings} to use arrows with markings}{}}{}%
    \pgfkeysalso{/tikz/postaction = {
      /tikz/decorate,
      /tikz/decoration={markings, mark = at position 0.5 with {#1}}}
    }
  },
}
\makeatother
\usepackage{calc}
\usepackage{accents}

\usepackage[headings]{fullpage}
\usetikzlibrary{babel}
\usetikzlibrary{positioning,calc}
\usepackage{varioref}
\usepackage[pdfpagelabels, pdftex]{hyperref}
\hypersetup{
  pdftitle={Largeness of the wild locus},
  pdfauthor={Katharina H\"{u}bner and Michael Temkin},
  pdfsubject={},
  pdfkeywords={},
  colorlinks=true,    
  linkcolor=red,     
  citecolor=blue,     
  filecolor=blue,      
  urlcolor=blue,       
  breaklinks=true,
  bookmarksopen=true,
  bookmarksnumbered=true,
  pdfpagemode=UseOutlines,
  plainpages=false,
  unicode=true
  }
\usepackage[capitalise]{cleveref}
\usepackage[inline]{showlabels}
\usepackage{changes}  
\usepackage{bbold} 
\newcommand{\Aut}{\mathrm{Aut}}

\newcommand{\m}{\mathfrak{m}}

\newcommand{\Z}{\mathds Z}

\newcommand{\Q}{\mathds Q}
\newcommand{\A}{\mathds A}
\newcommand{\D}{\mathds D}
\newcommand{\I}{\mathds I}

\newcommand{\R}{\mathds R}

\newcommand{\reduced}{\mathrm{red}}
\newcommand{\wild}{n^w}

\renewcommand{\P}{{\mathbb P}}

\newcommand{\disc}{\mathit{disc}}
\newcommand{\an}{\mathit{an}}

\newcommand{\cH}{{\mathcal H}}

\newcommand{\cO}{{\mathcal O}}
\newcommand{\cP}{{\mathcal P}}

\newcommand{\cR}{{\mathcal R}}

\newcommand{\cT}{{\mathcal T}}

\newcommand{\cX}{{\mathcal X}}

\newcommand{\liso}{\mathrel{\hbox{$\longrightarrow$} \kern-2.4ex\lower-1ex\hbox{$\scriptstyle\sim$}\kern1.7ex}}

\def\:{{\colon}}
\def\alp{\alpha}

\makeatletter
\newcommand{\zerounderset}[3][\mathord]{%
  #1{\vtop{
    \let\\\cr
    \baselineskip\z@skip\lineskip.25ex
    \ialign{\hidewidth$##$\hidewidth\crcr
      \omit$#3$\cr
      #2\crcr
    }%
  }}%
}
\makeatother

\newtheoremstyle{alexthm}
  {}
  {}
  {\sl }
  {}
  {\bf}
  {.}
  {.5em}
  {}
\theoremstyle{alexthm}

\newtheorem{theorem}{Theorem}[section]
\newtheorem*{theorem*}{Theorem}
\newtheorem{corollary}[theorem]{Corollary}
\newtheorem{proposition}[theorem]{Proposition}
\newtheorem{lemma}[theorem]{Lemma}
\newtheorem*{lemma*}{Lemma}

\newtheoremstyle{alexdef}
  {}
  {}
  {\rm }
  {}
  {\bf}
  {.}
  {.5em}
  {}
\theoremstyle{alexdef}
\newtheorem*{example*}{Example}
\newtheorem{example}[theorem]{Example}

\newtheorem{definition}[theorem]{Definition}
\newtheorem*{question*}{Question}

\DeclareMathOperator{\Spec}{\mathrm{Spec}}

\DeclareMathOperator{\Spa}{\mathrm{Spa}}

\DeclareMathOperator{\RZ}{\mathrm{RZ}}
\DeclareMathOperator{\supp}{\mathrm{supp}}

\DeclareMathOperator{\ch}{char}

\DeclareMathOperator{\Sch}{Sch}
\DeclareMathOperator*{\colim}{colim}

\DeclareMathOperator{\ad}{\textit{ad}}
\DeclareMathOperator{\trdeg}{\textit{trdeg}}

\definecolor{darklimegreen}{RGB}{31,142,8}

\begin{document}

\hfuzz=4pt
\title{Largeness of the wild locus}
\author{Katharina H\"{u}bner and Michael Temkin}
\date{\today}
\thanks{This work was supported by ISF Grant 1203/22. The first author acknowledges support by Deutsche Forschungsgemeinschaft (DFG) through the Collaborative Research Centre TRR 326 "Geometry and Arithmetic of Uniformized Structures", project number 444845124, and the second author acknowledges support by ERC Consolidator Grant 770922 - BirNonArchGeom, a grant from the School of Mathematics at the IAS, and support of MPIM-Bonn}

\begin{abstract}
 We show that the wild locus of an étale morphism of smooth adic spaces over an affinoid field is either empty or contains a divisorial point.
 As a consequence, we establish the equivalence of curve tameness and valuation tameness without using resolution of singularities.
 Moreover, given a finite étale map of regular schemes $f \colon Y \to X$ that is curve tame over a base~$S$, we construct a compactification $\bar{f} \colon \overline{Y} \to \overline{X}$ over~$S$ that is numerically tame.
\end{abstract}

\maketitle
\tableofcontents

\section{Introduction}

This research grew out of a question by Hélène Esnault on numerical tameness.
In order to understand this question let us go back to the beginnings of the investigation of tame covering spaces.
In \cite{SGA1}, and more thoroughly in \cite{GM71}, Grothendieck and Murre study covering spaces of a connected regular scheme~$\overline{X}$ that are tamely ramified along a fixed normal crossing divisor $D \subseteq \overline{X}$.
More precisely, these are integral morphisms to~$\overline{X}$ whose restriction to $X := \overline{X} \setminus D$ is finite étale, and such that the extension of function fields is tamely ramified with respect to all discrete valuations corresponding to the prime divisors in~$D$.
The underlying reason why this leads to a well behaved tame fundamental group $\pi_1^t(\overline{X},D)$ is Abhyankar's lemma telling us that étale locally on~$\overline{X}$, tame coverings of~$\overline{X}$ are obtained by extracting $n$-th roots of the functions defining the prime divisors in~$D$ where~$n$ is prime to the residue characteristic.
This implies that tame coverings of~$\overline{X}$ are not only tame at the generic points of~$D$, but at all points of~$D$.

In many situations the perspective on tame coverings is slightly shifted in the sense that we start with a regular scheme~$X$ over some base~$S$ and want to study finite étale morphisms $Y \to X$ that are tame over~$S$.
If we can choose a regular compactification~$\overline{X}$ of~$X$ over~$S$ whose boundary $D = \overline{X} \setminus X$ is a normal crossing divisor, tameness of a finite étale covering $Y\to X$ means precisely that it extends to a covering $\overline{Y}\to\overline{X}$ that is tamely ramified along~$D$.
Luckily this notion of tameness is independent of the choice of the compactification.
But unfortunately, we cannot guarantee the existence of nice compactifications without using resolutions of singularities, which is still a conjecture in positive and mixed characteristic.

If we want a theory of tame coverings that does not rely on a conjecture, we have to leave the paradise of regular schemes and normal crossing divisors.
Here the path splits into two directions.
The idea of the first approach is to define tameness of a finite étale covering $Y \to X$ without using a compactification.
This resulted in various notions of tameness, among them the concept of curve tameness (see \cite{Wie08} for the introduction of the concept and \cite{KeSch10} for a systematic treatment, the idea of testing a property on all curves goes back to Deligne, who defined regular singular connections in this way) requiring all base changes to curves to be tame.
The underlying idea is that on curves (which have to be defined with care in case~$S$ is not a field) it is clear what it means to be tame because curves have unique regular compactifications, whose boundary is a finite number of closed points in the regular locus, i.e. a divisor with normal crossings.
Another idea in this direction is the concept of divisor tameness where we meassure tameness with respect to all prime divisors showing up in some compactification.
In \cite{KeSch10} Kerz and Schmidt show that curve tameness and divisor tameness (and some more notions) are equivalent.
Slightly different in nature is the notion of valuation tameness where we use \emph{all} $S$-valuations of the function field of~$X$ to define tame covering spaces.
In contrast to curve tameness and divisor tameness, the relevant valuations may have a complicated geometric meaning and correspond to an infinite sequence of local modifications.
It is clear from the definition that valuation tameness implies curve and divisor tameness.
The converse direction, however, is hard and \cite{KeSch10} establishes it only in the presence of resolutions of singularities.

The second approach fixes a compactification~$\overline{X}$ and studies finite Galois covers of~$X$ that extend to~$\overline{X}$ in such a way that the orders of all inertia groups at all points of~$\overline{X}$ are prime to the residue characteristics.
These covers are called \emph{numerically} tame with respect to~$\overline{X}$ and this tameness depends on the choice of the compactification.
Numerical tameness is less valuation theoretic but rather scheme theoretic and plays a prominent role in the study of local systems.
If a Galois cover of~$X$ is numerically tame with respect to some compactification~$\overline{X}$, we can deduce that it is valuation tame.
In case~$\overline{X}$ is a normal crossing compactification of~$X$, any valuation tame cover of~$X$ is numerically tame with respect to~$\overline{X}$ by \cite[Theorem~5.4]{KeSch10}.
Unfortunately, we know how to construct normal crossing compactifications only by using resolutions of singularities.
Without assuming this conjecture, we cannot expect to find a compactification~$\overline{X}$ such that all valuation tame (or curve or divisor tame) covers of~$X$ are numerically tame with respect to~$\overline{X}$.

\begin{question*}[Hélène Esnault]
 Let~$S$ be a noetherian scheme and~$X$ a normal, connected scheme of finite type over~$S$.
 Given a curve tame Galois cover $f \colon Y \to X$, can we find a compactification~$\overline{X}$ of~$X$ over~$S$ such that~$f$ is numerically tame with respect to~$\overline{X}$?
\end{question*}

In \cref{compnumerical} we answer this question in the affirmative under some mild assumptions on~$S$:
We require~$S$ to be integral, pure dimensional, separated, and of finite type over an excellent Dedekind scheme.
The idea behind the proof is to use a compactness argument on the limit over all compactifications -- the relative Riemann Zariski space $\RZ_X(S)$ studied in \cite{Tem11}.
But in order to have access to this compactness argument we have to meassure tameness with respect to all valuations corresponding to points in $\RZ_X(S)$ (not just those coming from points on curves as in the definition of curve tameness).
For technical reasons we work with the adic space $\Spa(X,S)$ instead of $\RZ_X(S)$ enabling us to use the full strength of nonarchimedean geometry.
So in line with \cite{HueAd} and \cite{tamepi1}, we use the following notion of tameness:

\begin{definition}
	An étale morphism of schemes $Y \to X$ over~$S$ is \emph{tame with respect to}~$S$ if the étale morphism of adic spaces $\Spa(Y,S) \to \Spa(X,S)$ is tame, i.e., all residue field extensions are tame extensions of valued fields.
\end{definition}

Now the difficult question is whether curve-tameness is equivalent to our notion of tameness.
Assuming resolution of singularities, this question is not so hard to answer with the methods of \cite{KeSch10} telling us that the two notions are indeed equivalent.
However, for an unconditional result we cannot work with nice compactifications and accordingly our approach is valuation theoretic.
We do not prove directly that curve tameness implies tameness but rather take a detour via divisor tameness, which is equivalent (unconditionally) to curve tameness for regular schemes by \cite{KeSch10}.
Prime divisors in the boundary of a compactification of~$X$ over~$S$ correspond to \emph{divisorial points} of $\Spa(X,S)$ (\cref{description_divisorial_points_Spa}), so we need to show that the wild locus over $\Spa(X,S)$ is either empty or contains a divisorial point.
After a process of analytification discussed in \cref{sect:analytification}, we are led to study the wild locus of an étale morphism of adic spaces locally of finite type over an analytic field.
Our main structural result is the following:

\begin{theorem}[\cref{R_f_divisorial_point}] \label{main theorem}
	Let $f \colon Y \to  X$ be an étale morphism of smooth adic spaces over an affinoid field $(k,k^+)$ of finite rank.
	Then the wild locus $R_f \subseteq Y$ is either empty or contains a divisorial point.
\end{theorem}

The essential part of this claim is when $X$ and $Y$ are curves since the general case reduces to this using curve fibrations (see \cref{sect:wild_higher_dim}). We emphasize that the argument only works if we study curves over arbitrary affinoid fields -- of higher rank and not algebraically closed -- as these show up as fibers in a fibration even if we start with a rigid space over an algebraically closed nonarchimedean (i.e. rank one) field. The structure of curves over arbitrary affinoid fields has not been systematically studied in
the literature. We closed this gap with our companion paper \cite{adiccurves} where we developed the foundations to describe curves as metric
graphs in the same spirit as it is done for Berkovich spaces with the difference that the edges of the graph are not parametrized by the real
numbers but by \emph{rangers}, which can be thought of as a completion of the value group of the base field. In particular, this allows
us to make sense of paths between points on a curve and to study segments (small paths parametrized by rangers) at a point.

Our basic invariant used to control the wild locus is surprisingly simple -- it is just the {\em wild degree} of the residue field extension $k(y)/k(x)$. Since every constructible segment contains a divisorial point, the existence of divisorial points in the wild locus follows from our main theorem on étale morphisms of curves:

\begin{theorem}[\cref{wild_degree_constant_segments}]
	Let $f \colon Y \to X$ be an étale morphism of smooth curves over an affinoid field $(k,k^+)$ of finite rank.
	Then the wild degree is constant on sufficiently small constructible segments of~$Y$.
\end{theorem}

Roughly, for the proof we have to study two cases: arbitrary étale maps of smooth curves over algebraically closed fields and finite étale maps of the affine line induced by separable base field extensions.
The key input for the first case (treated in \cref{sect:alg_closed_ramification_locus}) is a structural result on the wild locus of a finite morphism of Berkovich curves over an algebraically closed field in \cite{CTT16}.
There the authors show that for a degree~$p$ morphism, the wild locus is radial around a suitable skeleton.
In particular, if the morphism is wild at a given point, it stays wild when we walk along a segment towards the skeleton.
For the second case (treated in \cref{sect:wild_degree_general}) we first study the wild degree over the interval $[0,\infty)$ of the affine line.
This suffices to construct segments where the wild degree is constant starting at some monomial point.
For arbitrary points we use a symmetry argument.

Finally, we want to stress that the above constructibility results in the study of the wild locus will be seriously refined in upcoming works by studying a quantitative invariant measuring wildness of the extension, the log different. We do not know to what extent the log different governs the shape of the wild locus extending the results for Berkovich curves over an algebraically closed field (see \cite{CTT16}), but we will prove that at the very least it is piecewise monomial on skeletons of Berkovich spaces and pathwise continuous on curves also at points of type four. If $f\:Y\to X$ is an \'etale morphism of adic spaces over an algebraically closed base field with a non-empty wild locus, then this locus contains a divisorial point $y$ by Theorem~\ref{main theorem}. In this case, the log different is strictly smaller than 1 at $y$ and can be used to study the wild locus near $y$ with a much better precision than the wild degree. However, in general there exist wild extensions undetected by the log different (for example, almost unramified extensions of non-archimedean fields), and the wild degree seems to be the only reasonable invariant which controls them.

\subsection*{Aknowledgements}
We want to thank Hélène Esnault whose question initiated this research.
We also thank her for various discussions and for taking a look at the manuscript and giving helpful comments that improved the exposition.

\section{Preliminaries}

\subsection{The base field}

We mainly work over an affinoid field $(k,k^+)$.
So $k$ is a field with a valuation $|\cdot|$ and valuation ring~$k^+$.
The topology on~$k$ can be either the valuation topology or the discrete topology.
In the first case~$k^+$ is required to be microbial and we call $(k,k^+)$ \emph{analytic} and in the second case \emph{discrete}.
We do not require~$k^+$ to be of rank~$1$.
Hence, $\Spa(k,k^+)$ might have more than one point.

In the analytic case the subring $k^\circ \subseteq k$ of powerbounded elements is a valuation ring of rank~$1$ containing~$k^+$.
It is the localization of $k^+$ at a topologically nilpotent element~$\varpi$.
We set
\[
 k^\succ := k^+/\m_{k^+},
\]
i.e., $k^\succ$ is the residue field of~$k^+$.
Throughout this article the characteristic of $k^\succ$ is $p > 0$.

We write $|\cdot|$ for the valuation corresponding to~$k^+$.
The residue field $k(x)$ at a point $x$ of an adic space~$X$ locally of finite type over $(k,k^+)$ is an extension of~$k$ of finite topological transcendence degree.
For an element $a \in k(x)$ we denote its valuation by $|a(x)|$.
We write $k(x)^+$ for the corresponding valuation ring and $k(x)^\succ$ for its residue field, which we also call \emph{specialization field} of~$x$ in order to distinguish it from the residue field~$k(x)$.

\subsection{The wild degree}

\begin{definition}
 Let $(L,w)/(K,v)$ be a finite extension of valued fields.
 Its \emph{wild degree} $\wild(L/K)$ is the degree of the corresponding extension of maximal tame extensions $L^t/K^t$ (taken with respect to any extension of~$w$ to an algebraic closure of~$L$).
\end{definition}

Note that the wild degree is independent of the choice of algebraic closure of~$L$ and of an extension of~$w$ to this algebraic closure.
In order to apply this concept to the residue field extensions of étale morphisms, we need the following observation.

\begin{lemma} \label{etale_finite_fibers}
 Let $f\colon  Y \to X$ be an étale morphism of adic spaces and $y \in Y$.
 Set $x = f(y)$.
 Then the extension of completed residue fields $\widehat{k(y)}/\widehat{k(x)}$ is a finite separable extension of valued fields.
\end{lemma}

\begin{proof}
 By \cite[Corollary~1.7.3~iii)]{Hu96}, $f$ is locally around~$y$ of the form $\Spa(B,B^+) \to \Spa(A,A^+)$, where $A \to  B$ is étale in the algebraic sense.
 We consider the supports $\supp y$ and $\supp x$ of $y$ and~$x$ in $\Spec B$ and $\Spec A$, respectively.
 Then $k(\supp y)/k(\supp x)$ is finite separable because it is a residue field extension of $B/A$.
 The field $k(\supp y)$ is naturally embedded into~$k(y)$ as a dense subfield and similarly for~$k(x)$.
 This implies that the extension of completions $\widehat{k(y)}/\widehat{k(x)}$ is finite separable.
\end{proof}

\begin{definition}
 Let $f\colon Y \to X$ be an étale morphism of adic spaces and $y \in Y$ a point with image $x \in X$.
 The \emph{wild degree} $\wild_{f}(y)$ of~$f$ at~$y$ is the wild degree of the completed residue field extension $\widehat{k(y)}/\widehat{k(x)}$.
\end{definition}

 Equivalently  we can define the wild degree of~$f$ at~$y$ as the degree of the corresponding morphism of tame henselizations
 \[
  f_{\bar{y}}^t \colon  Y_{\bar{y}}^t \longrightarrow X_{\bar{y}}^t.
 \]
 at any geometric point~$\bar{y}$ over~$y$.

\begin{lemma} \label{wd_multiplicative}
 Let
 \[
  Z \overset{g}{\longrightarrow} Y \overset{f}{\longrightarrow} X
 \]
 be étale morphisms of adic spaces and~$z \in Z$ a point mapping to $y \in Y$ and $x \in X$.
 Then
 \[
  \wild_{f \circ g}(z) = \wild_g(z) \cdot \wild_f(y).
 \]
\end{lemma}

\begin{proof}
 Let~$\bar{z}$ be a geometric point over~$z$ inducing geometric points~$\bar{y}$ and~$\bar{x}$ over~$y$ and~$x$.
 Then the result follows from the multiplicativity of the degrees of the finite morphisms
 \[
  Z_{\bar{z}}^t \overset{g_{\bar{z}}}{\longrightarrow} Y_{\bar{y}}^t \overset{f_{\bar{y}}}{\longrightarrow} X_{\bar{x}}^t.
 \]
\end{proof}

\begin{lemma} \label{wd_upper_semicont}
 Let $f\colon  Y \to X$ be an étale morphism of adic spaces.
 Then the wild degree $y \mapsto \wild_f(y)$ is upper semicontinuous on~$Y$
\end{lemma}

\begin{proof}
 Let $\bar{y} \to Y$ be a geometric point over~$y$.
 Then
 \[
  f_{\bar{y}}^t\colon  Y_{\bar{y}}^t \to X_{\bar{y}}^t
 \]
 is finite of degree $\wild_f(y)$.
 Therefore, there exist tame neighborhoods $X' \to X$ and $Y' \to Y$ of~$\bar{y}$ and a diagram
 \[
  \begin{tikzcd}
   Y'	\ar[r,"\text{tame}"]	\ar[d,"f'"']	& Y	\ar[d,"f"]	\\
   X'	\ar[r,"\text{tame}"]					& X,
  \end{tikzcd}
 \]
 such that $f'$ is finite of degree $\wild_f(y)$.
 In particular, the wild degree of~$f'$ at any point is less or equal to $\wild_f(y)$.
 Hence, for every $y_0$ in the image of $Y' \to Y$ we have
 \[
  \wild_f(y_0) \le \wild_f(y).
 \]
 Since the image is open, the assertion follows.
\end{proof}

\begin{lemma} \label{sandwich}
 Let
 \[
  Z \overset{g}{\longrightarrow} Y \overset{f}{\longrightarrow} X
 \]
 be étale morphisms of adic spaces over $\Spa(k,k^+)$, $s \in \Spa(k,k^+)$, and $I \subseteq Z$ a subset of~$Z$ for which~$f$ and~$g$ restrict to homeomorphisms
 \[
  I \overset{g|_I}{\longrightarrow} g(I) \overset{f|_{g(I)}}{\longrightarrow} (f \circ g)(I).
 \]
 Then $\wild_{f \circ g}$ is constant on~$I$ if and only if $\wild_g$ is constant on~$I$ and $\wild_f$ is constant on $g(I)$.
\end{lemma}

\begin{proof}
 Since~$I$ and~$g(I)$ are connected, it suffices to show the equivalence with ``constant'' replaced by ``locally constant''.
 Let~$z \in I$ be a point mapping to $y \in Y$ and $x \in X$.
 By the upper semicontinuity of the wild degree (\cref{wd_upper_semicont}) there is an open neighborhood $U \subseteq Z$ of~$z$ such that
 \begin{align*}
  \wild_{f \circ g}(u)& \le \wild_{f \circ g}(z), \quad \wild_g(u) \le \wild_g(z) \quad \forall u \in U,	\\
  &\text{and} \quad \wild_f(v) \le \wild_f(y) \quad \forall v \in g(U).
 \end{align*}
 (Note that $g(U)$ is open by \cite[Proposition~1.7.8]{Hu96}.)
 Moreover, by the multiplicativity of the wild degree (\cref{wd_multiplicative}) we have the identity
 \[
  \wild_{f \circ g}(u) = \wild_g(u) \circ \wild_f(g(u))
 \]
 for all $u \in U$.
 These two facts together tell us that for any open neighborhood $V \subseteq U$ of~$z$ the wild degree $\wild_{f \circ g}$ is constant on $V \cap I$ if and only if $\wild_g$ is constant on $V \cap I$ and $\wild_f$ is constant on $g(V \cap I)$.
 By our assumptions $g$ induces a homeomorphism $I \to g(I)$.
 In particular this implies that $g(V \cap I)$ is open in $g(I)$ and the opens of this form give a basis of neighborhoods for~$y$ in~$I$.
 This gives the desired statements about the local constancy of the wild degree.
\end{proof}

\subsection{Dimension of adic spaces}

Let $(k,k^+)$ be an affinoid field and $f\colon  Y \to X$ a morphism of adic spaces locally of ${}^+$weakly finite type over $(k,k^+)$.
The dimension of~$f$ is defined as
\[
 \dim f := \supp\{tr.c_{k(x)}k(y) \mid y \in Y,~ x = f(y)\},
\]
where $tr.c$ denotes the topological transcendence degree ($tr.c$ stands for continuous transcendence degree)
In particular, we define the dimension of~$X$ as
\[
 \dim X:= \dim(X \to \Spa(k,k^+))
\]
and the dimension of~$X$ at a point $x \in X$ as
\[
 \dim_x X := \inf_{x \in U \subset X} \dim U,
\]
where the infimum is taken over all open neighborhoods of~$x$.

The dimension as defined above is denoted $\mathrm{dim.tr}$ in \cite{Hu96} (see Definition~1.8.4 of loc.cit).
It is also possible to define the dimension in terms of chains of vertical specializations as is worked out for analytic adic spaces in \cite[\S~1.8]{Hu96}.
In particular, Corollary~1.8.7 in loc. cit. states that the two approaches to define the dimension coincide for analytic adic spaces locally of finite type over $(k,k^+)$.
If an adic space~$X$ is not topologically of finite type over $(k,k^+)$, the dimension in terms of transcendence degrees might be bigger than the dimension in terms of chains of vertical specializations.
As we are only interested in spaces topologically of finte type, it is sufficient for our purposes to work with the topological transcendence degree.

\begin{lemma} \label{dim_etale}
 Let $f\colon Y \to X$ be an étale morphism of adic spaces locally of finite type over $(k,k^+)$.
 Then $\dim f = 0$ and $\dim_{f(y)} X = \dim_y Y$ for every $y \in Y$.
\end{lemma}

\begin{proof}
 By \cref{etale_finite_fibers} we know that $\widehat{k(y)}/\widehat{(x)}$ is finite for all $y \in Y$ and $x = f(y)$.
 This implies $\dim f = 0$ and moreover that for all $y \in Y$, $x= f(y)$ we have
 \begin{equation} \label{compare_toptrdeg}
  tr.c_k k(x) = tr.c_k k(y).
 \end{equation}
 Since $f$ is étale, it is open (\cite[Proposition~1.7.8]{Hu96}).
 The open neighborhoods of~$x$ of the form $f(V)$ for an open neighborhood~$V$ of~$y$ are cofinal among all open neighborhoods of~$x$.
 Furthermore $\dim f(V) = \dim V$ by (\ref{compare_toptrdeg}).
 This implies that $\dim_{f(y)} X = \dim_y Y$.
\end{proof}

\subsection{Divisorial points}

For an extension $(K,K^+)/(k,k^+)$ of valued fields we define the following quantities
\[
 F_{K/k} := \trdeg_{k^\succ}(K^\succ), \qquad E_{K/k} := \dim_\Q(\lvert K^\times \rvert^\Q/\lvert k^\times \rvert^\Q).
\]
They are related by Abhyankar's inequality
\[
 tr.c_k(K) \ge F_{K/k} + E_{K/k}.
\]
This should be well known but seems not to be recorded in the literature in this form.
It follows by applying the Abhyankar inequality for the algebraic transcendence degree (see \cite[Theorem~3.4.3]{EP2005}) to a dense intermediate field $k \subseteq K_0 \subseteq K$ with $\trdeg_k(K_0) = tr.c_k(K)$ and noting that $F_{K_0/k} = F_{K/k}$ and $E_{K_0/k} = E_{K/k}$.

\begin{definition}
 An extension of valued fields $(K,K^+)/(k,k^+)$ is \emph{divisorial} if
 \[
  tr.c_k(K) = F_{K/k} < \infty.
 \]
\end{definition}

For a divisorial extension of valued fields $(K,K^+)/(k,k^+)$ Abhyankar's inequality is automatically an equality.
Moreover, both fields are of the same rank.

We now consider an adic space~$X$ that is of ${}^+$weakly finite type over an affinoid field $(k,k^+)$.
For $s \in S:= \Spa(k,k^+)$ we study the fiber~$X_s$ over~$s$.
In this setting we use the notation
\[
 F_x = F_{k(x)/k(s)} = \trdeg_{k(s)^\succ}(k(x)^\succ), \quad E_x := E_{k(x)/k(s)} = \dim_\Q(\lvert k(x)^\times \rvert^\Q/\lvert k(s)^\times\rvert^\Q).
\]

\begin{definition}
 A point $x \in X_s$ is called \emph{divisorial} if $tr.c_k k(x) = \dim_x X$ and either
 \begin{enumerate}[(i)]
  \item $s$ is nontrivial and $(k(x),k(x)^+)/(k(s),k(s)^+)$ is divisorial or
  \item $(k,k^+)$ carries the discrete topology, $s$ is the trivial valuation, and $\trdeg_{k(s)}(k(x)) = F_x + 1$.
 \end{enumerate}
 We also refer to the first kind as \emph{vertically divisorial} and the second kind as \emph{horizontally divisorial}.
\end{definition}

The intuition behind the nomenclature is the following.
On a formal model of~$X$ over~$k^+$ a vertically divisorial point $x \in X_s$ corresponds to a divisor contained in the fiber over the center of~$s$.
Horizontally divisorial points only exist if~$X$ is discretely ringed, i.e. locally of the form $X = \Spa(\cX,\cT)$ for a morphism $\cX \to \cT$ of schemes where~$\cX$ is of finite type over~$k$ and $\cT$ is of finite type over~$k^+$.
In this situation a horizontally divisorial point corresponds to a divisor of~$\cX$.

The case of curves over $(k,k^+)$ has been treated in \cite{adiccurves}.
For curves, vertically divisorial points $x \in X_s$ are characterized by $F_x = 1$ and are called \emph{divisorial points} in loc. cit.
Horizontally divisorial points are precisely the unbounded points of rank~$1$.

\begin{lemma} \label{divisorial_finite_ext}
 Let $(L,L^+)/(K,K^+)$ be a finite extension of valued fields over $(k,k^+)$.
 Then $L/k$ is divisorial if and only if $K/k$ is divisorial.
\end{lemma}

\begin{proof}
 Both $F_{K/k}$ and $tr.c_k(K)$ are invariant under finite extensions.
\end{proof}

\begin{lemma} \label{divisorial_etale}
 Let $f\colon Y \to X$ be an étale morphism of adic spaces.
 Then $y \in Y$ is divisorial if and only if $f(y)$ is divisorial.
\end{lemma}

\begin{proof}
 By \cref{etale_finite_fibers} the completed residue field extensions $\widehat{k(y)}/\widehat{k(x)}$ are finite for all $y \in Y$ and $x = f(y)$.
 Therefore, the assertion follows from \cref{dim_etale} and \cref{divisorial_finite_ext}.
\end{proof}

\subsection{Segments}

For a smooth curve~$X$ over~$k$ we consider the \emph{sheaf of values}
\[
 \Gamma_X = \cO_X^\times/(\cO_X^+)^\times.
\]
If~$k$ is algebraically closed, it is a useful invariant to describe the metric structure of~$X$.
In the general case, however, it is better to work with the \emph{geometric sheaf of values} $\overline{\Gamma}_X$.
It is constructed by noting that the sheaf of values on the base change $X_{\bar{k}}$ of~$X$ to an algebraic closure of~$k$ descends to a sheaf on~$X$.
This is the sheaf $\overline{\Gamma}_X$ (see \cite[\S~4.3.2]{adiccurves}).

We fix a point $s \in S:=\Spa(k,k^+)$ and set
\[
 \overline{\Gamma} := \lvert k(s)^\times \rvert^\Q.
\]
By \cite[Theorem~4.7.9]{adiccurves}, the fiber~$X_s$ together with~$\overline{\Gamma}_X$ forms a $\overline{\Gamma}$-graph with pinched leaves.
Its vertices are the divisorial points, its edge points are the cut and infinitesimal points and its pinched leaves are the unibranch and classical points.

In contrast to the Berkovich world, the edges of this graph are not parametrized by the real numbers but by the set of rangers~$\cR_\Gamma$.
This is an ordered set containing~$\Gamma$.
Equipped with the topology generated by the intervals of the form $[\gamma,\infty]$ and $[-\infty,\gamma]$ for $\gamma \in \Gamma$, $\cR_\Gamma$ is connected, compact, and contains~$\Gamma$ as a dense subset.
The rangers are explicitly constructed as a set of equivalence classes of pairs $(\Gamma',\gamma')$, where~$\Gamma'$ is a totally ordered abelian group containing~$\Gamma$ and~$\gamma'$ is an element of~$\Gamma'$.
The equivalence is induced by inclusions of totally ordered abelian groups~$\Gamma'$ identifying the respective elements~$\gamma'$.

Now let us get back to the graph $(X_s,\overline{\Gamma})$.
It makes sense to speak of subsegments of~$X_s$.
A large part of this article will be concerned with proving that the wild degree of an étale morphism $Y \to X$ stays constant when we walk along a small segment of~$Y_s$.
However, for technical reasons we want to treat the fibers over varying $s \in S$ uniformly.
The technical framework for this is provided by the theory of \emph{valuative graphs} developed in \cite[Chapter~5]{adiccurves}.
A valuative graph~$Z$ is glued from rational subsets of the space $\Spa(\overline{\Gamma}[T]_\star,\overline{\Gamma}^+[T]_\star)$ of equivalence classes of continuous valuations of the augmented monoid $\overline{\Gamma}[T]_\star$ whose value is less or equal to~$1$ on $\overline{\Gamma}^+[T]$.
It is provided with the structure sheaf~$\cP_Z$ of piecewise linear functions and its subsheaf~$\cP_Z^+$ of integral piecewise linear functions.

We can enhance a finite subgraph~$Z_s$ of~$X_s$ whose endpoints are divisorial or classical to a quasi-compact valuative graph~$Z$ over $\Spa(k(s),k(s)^+)$ whose fiber over~$s$ recovers~$T_s$.
More precisely, we consider on the base change~$\overline{X}$ of~$X$ to the algebraic closure of~$k$, the sheaf of piecewise monomial functions $\cP_{\overline{X}} := \cO_{\overline{X}}/\cO_{\overline{X}}^{+\times}$ and its subsheaf~$\cP_{\overline{X}}^+ = \cO_{\overline{X}}^+/\cO_{\overline{X}}^{+\times}$ of integral piecewise monomial functions
Both~$\cP_{\overline{X}}$ and~$\cP_{\overline{X}}^+$ descend to sheaves~$\overline{\cP}_X$ and~$\overline{\cP}_X^+$ on~$X$.
Then $X^{\overline{\times}} := (X,\overline{\cP}_X,\overline{\cP}_X^+)$ is a monoided valuation space in the sense of~\cite[\S~5.3.1]{adiccurves}.
We can now study subspaces~$Z$ of $X^{\overline{\times}}$ that are valuative graphs.
Their fibers over some $s \in S$ will be finite subgraphs of~$X_s$ with divisorial or classical endpoints.
The structure sheaf~$\cP_Z$ of piecewise linear functions on~$Z$ recovers the restriction of the sheaf~$\overline{\Gamma}_X$ to~$Z$: $\overline{\Gamma}_X|_Z = \cP_Z^\times$.

We say that~$Z$ is a valuative graph \emph{over $s \in S$} if it is glued from intervals of the form $[a,b]_s$ (with always the same~$s$).
Unfortunately, this does not suffice to treat segments at unibranch points.
Recall from \cite[5.3.3]{adiccurves} the definition of the ray $[x,\infty]_s$ for a unibranch point $x \in \P^1_{\bar{k}}$.
It is of the form $\{x\} \cup \bigcup_{i \in I} [a_i,\infty]_s$, where~$I$ is the set of discs containing~$x$ with Gauß point~$a_i$.
By local uniformization \cite[Theorem~4.6.5]{adiccurves} and invariance under $h$-isomorphisms \cite[Proposition~5.3.6]{adiccurves}, we conclude that for a unibranch point~$x$ on a general smooth curve over~$(\bar{k},\bar{k}^+)$ all segments emanating from~$x$ are locally at~$x$ of the form $\{x\} \cup \bigcup_{i \ge j} [a_i,a_j]_s$ for some~$j \in I$.
Taking the quotient of the curve by the Galois action is well behaved by \cite[Lemma~4.3.11]{adiccurves} and thus segments at unibranch points of smooth curves over $(k,k^+)$ are of the same form.

\begin{definition}
 An \emph{extended valuative graph} over $s \in S$ in a smooth curve~$X$ over~$(k,k^+)$ is a monoided valuation subspace $Z \subseteq X^{\overline{\times}}$ that is locally isomorphic to intervals $[a,b]_s \subseteq \P^1_{\bar{k}}$, where~$b$ is divisorial and~$a$ is either divisorial, classical or unibranch.
\end{definition}

A \emph{constructible segment} in an extended valuative graph~$Z$ is an element of the boolean algebra generated by all intervals $[a,b]_s$ (for the fixed~$s$) as in the definition above.
Explicitly a constructible segment of~$Z$ can always be obtained by taking a quasi-compact open extended valuative subgraph $Z' \subseteq Z$ over~$s$ and removing for some of the divisorial endpoints $p \in Z'$ the subspaces $\Spa(k(p),k(p)^+)$.
Consequently, the endpoints of a constructible segment of~$Z$ are either divisorial, classical, infinitesimal, or unibranch.

\begin{definition}
 Let~$X$ be a smooth curve over an affinoid field~$(k,k^+)$.
 A \emph{constructible segment} of~$X$ is a constructible segment of a qcqs extended valuative subgraph of $(X,\cP_X,\cP_X^+)$.
\end{definition}


\begin{lemma} \label{segment_divisorial}
 Let $I$ be a constructible segment of a smooth curve~$X$ over $S= \Spa(k,k^+)$ and let $s$ be a point in the image of~$I$ in~$S$.
 Then~$I$ contains a divisorial point over~$s$.
\end{lemma}

\begin{proof}
 Let $x \in I$ be a point over~$s$.
 We reduce to the case where~$s$ is the closed point of~$S$ and~$I$ is isomorphic to a constructible subspace of the unit interval $\I_S = [0,1]_s = \Spa(\overline{\Gamma}[T]_\star,\overline{\Gamma}^+[T]_\star)$.
 These are finite unions of intervals of the form $[a,b]_s$ for divisorial rangers $a \le b \le 1$, $[a,b)_s$, $(a,b]_s$, and $(a,b)_s$, for divisorial rangers $a < b \le 1$, and $[0,b]_s$, $[0,b)_s$ for a divisorial ranger $b \le 1$.
 The divisorial points over~$s$ in these intervals correspond to divisorial rangers between $a$ and~$b$ (or $0$ and~$b$).
 They exist since~$\overline{\Gamma}$ is dense in~$\cR_{\Gamma}$.
%
%
\end{proof}
%
For a point $x \in X_s$, we call $I \subseteq X$ a \emph{segment at}~$x$ if~$I$ is a constructible segment over~$s$ containing~$x$.
If~$x$ is a divisorial point, the valuative graph $\Spa(\overline{k(x)},\overline{k(x)}^+)$ is a segment at~$x$.
For all other points~$x$, segments at~$x$ have to contain other points over~$s$ apart from~$x$.


\begin{proposition} \label{segment_homeomorphism}
 Let $f\colon  Y \to X$ be an étale morphism of smooth curves over $(k,k^+)$ and $s \in S = \Spa(k,k^+)$ a point.
 \begin{enumerate}[(i)]
  \item Let~$I$ be a constructible segment at a point $y \in Y_s$.
 	Then there is a constructible subsegment $J \subseteq I$ at~$y$ such that
 	\[
 	 f|_J \colon  J \longrightarrow f(J)
	\]
 	is a homeomorphism.
  \item If $I \subseteq X$ is a constructible segment over~$s$, the preimage $f^{-1}(I)$ is a constructible segment of~$Y$.
 \end{enumerate}
\end{proposition}

\begin{proof}
 (i). If~$y$ is not unibranch, this follows from \cite[Theorem~5.3.12]{adiccurves}.
 In case~$y$ is unibranch, we use local uniformization \cite[Theorem~4.6.5]{adiccurves} to reduce to the case where~$X$ and~$Y$ are $h$-discs.
 We may assume that~$I$ is of the form $\{y\} \cup \bigcup_{i \in I} [a_i,b]_s$, where~$b$ is the Gauß point of a rational disc~$\D_b$ over~$s$ containing~$x$ and~$I$ is the set of discs containing~$x$ and contained in~$\D_b$.
 Then the assertion follows by applying \cite[Proposition~5.3.8]{adiccurves} to the intervals $[a_i,b]_s$.

 (ii). If~$I$ does not contain unibranch points, the assertion again follows from \cite[Theorem~5.3.12]{adiccurves}.
 Otherwise we decompose~$I$ into a valuative graph and intervals of the form $[x,b]_s \subseteq [x,\infty]_s \subseteq \P^1_{\bar{k}}$ for a unibranch point~$x$ and a divisorial point~$b$.
 Then we can apply loc.cit. to the valuative graph and \cite[Proposition~5.3.8]{adiccurves} to the subintervals~$[a_i,b]_s$ of $[x,b]_s$ for~$a_i$ divisorial.
\end{proof}

Combining \cref{segment_homeomorphism} with \cref{sandwich} we obtain a very useful tool in the toolbox for our main theorem.

\subsection{The wild locus}

The overall goal of this article is to investigate the wild locus of an étale morphism of adic spaces.

\begin{definition}
 The \emph{wild locus} of an étale morphism $f \colon Y \to X$ of adic spaces is the set of all $y \in Y$ with image $x \in X$ such that the completed residue field extension $\widehat{k(y)}/\widehat{k(x)}$ is wildly ramified.
 We denote the wild locus by $R_f$.
\end{definition}

We know by \cite[Corollary~4.4]{HueAd} that the wild locus is a closed subset of~$Y$.
But this does not predict anything about the size or shape of the wild locus.
We expect it to be rather big in the sense that it should not contain isolated closed points.
In the presence of a nice formal model we have more control on the wild locus as the following example shows.

\begin{example}
 Let $(K,K^+)$ be a discrete valuation field and $X/K$ a smooth rigid space with a semistable formal model $\cX/K^+$.
 By Abhyankar's lemma, a finite étale covering $f \colon Y \to X$ is tame if and only if it is tame with respect to the valuations~$v_D$ corresponding to the irreducible components of the special fiber of~$\cX$.
 We can rephrase this in the following way:
 If the wild locus~$R_f$ is non-empty, it contains one of the divisorial points~$v_D$ as above.
\end{example}

If we do not have a nice formal model for our space (and we do not want to resort to resolutions of singularities, which is only conjecturally true, in order to construct such a model), it is not so easy to find a divisorial point in the wild locus.
For curves our strategy will be to start with any point in~$R_f$ and walk in a suitable direction along a segment.
We will show in sections~\ref{sect:alg_closed_ramification_locus} and~\ref{sect:wild_degree_general} that the morphism~$f$ stays wild on a small enough subsegment and this subsegment needs to contain a divisorial point.
In higher dimension we use fibrations in curves in order to reduce to the curve case (see section~\ref{sect:wild_higher_dim}).

\section{Passing to finite étale morphisms}

In this section we explain why it suffices to study the wild degree of \emph{finite} étale morphisms.

Remember from \cite[\S 6]{HueAd} that a Huber pair $(A,A^+)$ is called \emph{local} if~$A$ is local, $A^+$ is the preimage of a valuation ring of the residue class field of~$A$ (this implies that~$A^+$ is local as well), and the maximal ideal of~$A^+$ is open and bounded.
By \cite[Lemma~6.8]{HueAd}, an adic space is the spectrum of a local Huber pair if and only if it has a unique closed point that every other point specializes to.
By taking the intersection of all affinoid neighborhoods of a point $x$ in an adic space~$X$, we obtain a local adic space $X_x$ that we call the \emph{localization} of~$X$ at~$x$ (see \cite{HueAd}, \S~6.3).

A Huber pair $(A,A^+)$ is \emph{henselian} if it is local and~$A^+$ is henselian.
An adic space is called henselian if it is the adic spectrum of a henselian Huber pair.
This automatically implies that also~$A$ is henselian (see \cite[Proposition~6.10]{HueAd}).
For a point $x$ of an adic space~$X$ we obtain a henselian adic space $X_x^h$ by taking the limit over all pointed étale morphism $(Y,y) \to (X,x)$ such that the completed residue field extension $\widehat{k(y)}/\widehat{k(x)}$ is trivial (\cite[\S~6.3]{HueAd}).
We call $X_x^h$ the \emph{henselization} of~$X$ at~$x$.

\begin{lemma} \label{etale_henselian}
 Let $(A,A^+)$ be henselian and $f\colon  X \to \Spa(A,A^+)$ an étale and surjective morphism from a connected adic space~$X$.
 Then~$f$ is finite.
\end{lemma}

\begin{proof}
 Let $x \in X$ be a preimage of the closed point of $\Spa(A,A^+)$.
 Choose a connected affinoid open neighborhood $U= \Spa(B,B^+)$ of~$x$ in~$X$.
 By \cite[Corollary~1.7.3]{Hu96}, we may assume that $(B,B^+)$ is étale over $(A,A^+)$ in the algebraic sense, i.e. $A \to B$ is étale as a map of rings.
 As~$A$ is henselian, this implies that $B/A$ is finite and~$B$ is local (henselian).
 Denote by $B_0^+$ the integral closure of~$A^+$ in~$B$.
 Then $(B,B_0^+)$ is henselian being integral over $(A,A^+)$ (see \cite[Lemma~6.14]{HueAd}).
 We obtain the following diagram
 \[
  \begin{tikzcd}[column sep=tiny]
   \Spa(B,B^+)	\ar[rr,open]	\ar[dr,"\text{étale}"']	&				& \Spa(B,B_0^+)	\ar[dl,"\text{finite étale}"]	\\
														& \Spa(A,A^+).
  \end{tikzcd}
 \]
 Since $(B,B_0^+)$ is local and $\Spa(B,B_0^+) \to \Spa(A,A^+)$ is finite étale, the closed point of $\Spa(B,B_0^+)$ is the unique preimage of the closed point of $\Spa(A,A^+)$.
 But this point needs to coincide with $x \in Spa(B,B^+)$, whence $\Spa(B,B^+) = \Spa(B,B_0^+)$.

 If we do this construction for a different point $y \in X$ in the preimage of the closed point of $\Spa(A,A^+)$, we obtain another finite étale map
 \[
  \Spa(C,C^+) \to \Spa(A,A^+)
 \]
 for an open affinoid neighborhood neighborhood $\Spa(C,C^+)$ of~$y$.
 But then
 \[
  \Spa(B,B^+) \times_X \Spa(C,C^+) = \Spa(B,B^+) \cap \Spa(C,C^+)
 \]
 is finite étale over $\Spa(A,A^+)$ but not surjective.
 Hence the intersection is empty.
 But by assumption $X$ is connected, so there can be only one point of~$X$ mapping to the closed point of $\Spa(A,A^+)$.
\end{proof}

We say that an étale morphism of adic spaces $f\colon  Y \to X$ is \emph{tame} if all completed residue field extensions $\widehat{k(y)}/\widehat{(k(y)}$ are tamely ramified.

\begin{proposition} \label{red_fet}
 Let $f\colon Y \to X$ be an étale morphism and $y \in Y$ a point with $\ch^+(y) = p > 0$.
 Then there is a diagram
 \[
  \begin{tikzcd}
   Y'	\ar[r,"\text{tame}"]	\ar[d,"\text{finite étale}"']	& Y	\ar[d,"f"]	\\
   X'	\ar[r,"\text{tame}"']									& X
  \end{tikzcd}
 \]
 such that
 \begin{itemize}
  \item $y$ lies in the image of $Y' \to Y$,
  \item $Y' \to X'$ is a composition of (finitely many) finite étale morphisms of degree~$p$.
 \end{itemize}
\end{proposition}

\begin{proof}
 Let~$x$ denote the image of~$y$ in~$X$.
 Consider the basechange of~$f$ to the tame henselization of~$X$ at~$\bar{x}$:
 \[
  \begin{tikzcd}
   X_x^t \times_X Y	\ar[r]	\ar[d,"f_x^h"']	& Y	\ar[d,"f"]	\\
   X_x^t			\ar[r]					& X.
  \end{tikzcd}
 \]
 Choose a preimage $y_0 \in X_x^t \times_X Y$ of~$y$ and denote by $Y_0 \subseteq X_x^t \times_X Y$ the connected component containing~$y_0$.
 By \cref{etale_henselian} the morphism $Y_0 \to X_x^t$ is finite étale.
 Since $X_x^t$ does not admit any nontrivial tame covers, the morphism $Y_0 \to X_x^t$ is purely wild at~$y_0$.
 In particular, it is a $p$-extension and as such it factors into a composition of finite étale morphisms of degree~$p$.

 As finite étale morphisms are algebraically of finite presentation by \cite[Corollary~1.7.3]{Hu96}, the morphism $Y_0 \to X_x^t$ is the base change of a finite étale morphism $Y' \to X'$ of the same degree for some tame étale neighborhood $(X',x') \to (X,x)$.
\end{proof}

\section{The wild degree for curves over an algebraically closed field} \label{sect:alg_closed_ramification_locus}

In this section we study finite étale morphisms
\[
 f\colon  Y \longrightarrow X
\]
of degree~$p$ of smooth curves~$X$ and~$Y$ over an algebraically closed affinoid field~$(k,k^+)$.
We will prove by induction on the rank of~$k^+$ that the wild degree is constant on sufficiently small constructible segments.
The starting point is the rank one case, where we can use the structure theory of finite étale morphisms of Berkovich curves studied in \cite{CTT16}.

The most important invariant that is used in \cite{CTT16} to describe the ramification behaviour of an étale morphism $f_B \colon  Y_B \to X_B$ of nice Berkovich curves over the algebraically closed field~$k$ is the different~$\delta_{f_B}$.
Its value at a point $y_B \in Y_B$ mapping to $x_B \in X_B$ is the different of the completed residue field extension $\cH(y_B)/\cH(x_B)$.
For divisorial and cut points (type~$2$ and~$3$) the different is less than one if and only if the extension $\cH(y_B)/\cH(x_B)$ is wildly ramified.
By \cite[Corollary~4.1.8]{CTT16}, the different~$\delta_{f_B}$ is piecewise $\lvert k^\times \rvert$-monomial.

In order to describe the ramification locus of~$f_B$ we choose a skeleton $(\Gamma_{Y_B},\Gamma_{X_B})$ of~$f_B$ and denote by $\Gamma_B \subseteq \Gamma_{Y_B}$ the subgraph consisting of all points where~$f_B$ is wild.
Since all points on~$\Gamma_{Y_B}$ are cut or divisorial points, $\Gamma_B$ can be described as the subset of~$\Gamma_{Y_B}$ comprised of all points where the different is less than one.
Since $\Gamma_{Y_B}$ is a skeleton of~$Y_B$, the complement of~$\Gamma_{Y_B}$ is a disjoint union of closed constructible unit h-discs (see \cite[\S 4.7]{adiccurves}.
Every such disc corresponds to a branch of a divisorial point on~$\Gamma_{Y_B}$ pointing away from~$\Gamma_{Y_B}$. 
On the unit disc we can meassure the distance to the Gauß point using the radius function.
More precisely, the distance of a point~$z$ in the disc to the Gauß point is $-\log r_z$, where~$r_z$ is the radius of~$z$.
It equals~$\infty$ if~$z$ is a classical point and~$0$ if~$z$ is the Gauß point.

For a divisorial point $z_0 \in \Gamma_B$ we define the set $C(\Gamma_B,z_0,\delta_{f_B}(z_0)^{1/(p-1)})$ as the set of all points $z \in Y_B \setminus \Gamma_{Y_B}$ lying in a disc corresponding to a branch at~$z_0$ of radius $r_z  \ge \delta_{f_B}(z_0)^{1/(p-1)}$.
Then $C(\Gamma_B,\delta_{f_B}^{1/(p-1)})$ is defined as the union of~$\Gamma_B$ with all the sets $C(\Gamma_B,z_0,\delta_{f_B}(z_0)^{1/(p-1)})$.
It is a so called \emph{radial set}.

\begin{theorem}[\cite{CTT16}, Theorem~7.1.4] \label{Berkovich_ramification_locus}
 Suppose~$X_B$ and~$Y_B$ are compact, separated, and quasi-smooth and $f_B\colon  Y_B \to X_B$ is finite, generically étale of degree~$p$.
 Then the wild locus of $f_B$ equals the radial set $C(\Gamma_B,\delta_{f_B}^{1/(p-1)})$.
\end{theorem}

In fact for the points of radius precisely equal to $\delta_{f_B}^{1/(p-1)}$ we need to be more careful.
These points can be either divisorial or unibranch and the divisorial points of radius $\delta_{f_B}^{1/(p-1)}$ do not belong to the wild locus but the unibranch points to.
Using rangers the picture becomes clearer because the radius of a unibranch point would naturally be the ranger which is infinitesimally larger than $\delta_{f_B}^{1/(p-1)}$.
So the ramification locus becomes the radial set of all points of radius strictly larger than $\delta_{f_B}^{1/(p-1)}$.

We would like to transfer this result to the category of adic spaces.
The skeleton $(\Gamma_{Y_B},\Gamma_{X_B})$ of~$f_B$ comes from a skeleton $(\Gamma_Y,\Gamma_X)$ of~$f$.
In particular, the complement of~$\Gamma_Y$ is a disjoint union of constructible closed unit h-discs.
Let $\Gamma \subseteq \Gamma_Y$ be the closed subgraph consisting of all points where~$f$ is wild.
The different as it is defined is a function on the Berkovich space~$Y_B$.
We lift it to a function~$\delta_f$ on the adic space~$Y$ by setting
\[
 \delta_f(y) := \delta_{f_B}(y_B).
\]
Here~$y_B$ is the image of~$y$ in~$Y_B$.
Moreover, we embed~$\R$ into the set of rangers $\cR_{\lvert k^\times \rvert}$ in the obvious way such that we can make sense of inequalities of the form $\delta_f(y) \le r$ for $r \in \R_{\lvert k^\times \rvert}$.
This is quick and dirty and does not reflect the ramification behaviour at infinitesimal points.
We will investigate a refined different function in upcoming work but for our present purposes this coarse invariant suffices.

On the unit disc $\D_k = \Spa(k \langle T \rangle,k^\circ \langle T \rangle)$ we define the radius function as follows:
\begin{align*}
 r \colon \D_k	& \longrightarrow \cR_{\lvert k^\times \rvert}	\\
 x				& \mapsto \inf_{c \in k} \lvert T - c \rvert,
\end{align*}
where the infimum is taken in the set of rangers.
The main differences to the radius function on the Berkovich unit disc are twofold.
Firstly, we also assign a radius to infinitesimal points, which nonsurprisingly turns out to be an infinitesimal ranger.
Secondly, the radius of a unibranch point is not a principal ranger as one might expect but an infinitesimal or cut ranger.

Let us now assume that $f\colon  Y \to X$ is a finite generically étale morphism of degree~$p$ of proper smooh curves over $(k,k^\circ)$.
For a divisorial point $z_0 \in \Gamma$ we define the set $C(\Gamma,z_0,\delta_f(z_0)^{1/(p-1)})$ as the set of all points $z \in Y \setminus \Gamma_Y$ lying in a disc corresponding to a branch at~$z_0$ of radius
\[
 r(z) > \delta_f(z_0)^{1/(p-1)}.
\]
Note that these are intervals of the form $[z_0,w]$, where~$w$ is eiter an infinitesimal or a unibranch point.
As in the Berkovich setting we define the radial subset $C(\Gamma,\delta_f^{1/(p-1)})$ as the union of~$\Gamma$ with the sets $C(\Gamma,z_0,\delta_f(z_0)^{1/(p-1)})$.

\begin{theorem} \label{alg_closed_ramification_locus}
 Let $f \colon Y \to X$ be a finite, generically étale morphism of degree~$p$ of smooth, proper curves over the algebraically closed nonarchimedean field $(k,k^\circ)$ of residue characteristic~$p$.
 Then the wild locus of~$f$ equals $C(\Gamma,\delta_f^{1/(p-1)})$.
\end{theorem}

\begin{proof}
 A point $y \in Y$ of rank~$1$ is contained in the wild locus of~$f$ if and only if its image~$y_B$ in the Berkovich space~$Y_B$ is contained in the ramification locus of~$f_B$.
 Hence, \cref{Berkovich_ramification_locus} settles the question whether they lie in the wild locus for all points of rank~$1$.

 Let us now investigate the matter for an infinitesimal point~$y$.
 It has a unique generalization of rank~$1$, which is a divisorial point that we denote by~$y^\circ$.
 We write~$x$ and~$x^\circ$ for the images of~$y$ and $y^\circ$ in~$Y$.
 If~$y^\circ$ lies in the wild locus, so does~$y$ as the tame locus is open by \cite[Corollary~4.4]{HueAd}.
 We are left with studying the case where~$f$ is tame at~$y^\circ$, i.e. $\widehat{k(y^\circ)}/\widehat{k(x^\circ)}$ is unramified.
 The extension of specialization fields $k(y^\circ)^\succ/k(x^\circ)^\succ$ is a finite separable extension of one-dimensional function fields over~$k^\succ$.
 The corresponding morphism $f_{y^\circ}^\succ \colon  C_{y^\circ} \to C_{x^\circ}$ of smooth projective curves ramifies at the point corresponding to~$y$ if and only if~$y$ is contained in the wild locus.
 By the Riemann Hurwitz formula we have
 \[
  2(g(C_{y^\circ}) -2) = 2p(g(C_{x^\circ}) - 2) + \sum_{z \in C_{y^\circ}} (-\mathrm{slope}~\delta_{f_B}(z) + e_z -1),
 \]
 where the slope of~$\delta_{f_B}$ at~$z$ is defined as follows.
 We interpret~$z$ as an infinitesimal point of~$Y$ and as such it corresponds to a branch of~$Y_B$ emanating from $y_B = z_B$.
 Now~$\delta_{f_B}$ is a piecewise monomial function on~$Y_B$, so there is some point $z'_B$ on the branch of~$z$ such that~$\delta_{f_B}$ is monomial on $[y_B,z'_B]$.
 The slope of~$\delta_{f_B}$ at~$z$ is defined as the exponent of this monomial function.

 We know that~$f_{y^\circ}^\succ$ is wild at~$z$ if and only if the corresponding summand in the Riemann Hurwitz formula is greater than $e_z - 1$, or in other words if and only if the slope of~$\delta_{f_B}(z)$ is nonzero.
 Now by \cite[Theorem~7.1.4]{CTT16} and its proof the different has slope $\pm(p-1)$ at all infinitesimal points in the sets $C(\Gamma,z_0,\delta_f(z_0)^{1/(p-1)})$ and slope zero on the complement of $C(\Gamma,\delta_f^{1/(p-1)})$.
 This proves the theorem.
\end{proof}

\begin{corollary} \label{wild_all_intervals}
 Let $f\colon  Y \to X$ be a finite étale morphism of smooth curves of degree~$p$ over an algebraically closed nonarchimedean field $(k,k^\circ)$ of residue characteristic~$p$.
 Then the wild degree is constant on sufficiently small constructible segments.
\end{corollary}

\begin{proof}
 We just need to make sure we are in the postion to apply \cref{alg_closed_ramification_locus}.
 The question being local we can assume that~$X$ and~$Y$ are affinoid.
 Then we can extend~$f$ to a generically étale morphism
 \[
  \bar{f} \colon \bar{Y} \longrightarrow \bar{X}
 \]
 of smooth projective curves such that~$Y$ and~$X$ are open subspaces of~$\bar{Y}$ and~$\bar{X}$ and the restriction of~$\bar{f}$ to~$Y$ equals~$f$.
 By the description of the wild locus of~$\bar{f}$ in \cref{alg_closed_ramification_locus} it is now clear that the wild degree is constant on sufficiently small constructible segments.
\end{proof}

\cref{wild_all_intervals} is the key result for treating étale morphisms over algebraically closed fields.
Piecing together what we have done in the previous sections we can enhance it to a statement in full generality.

\begin{proposition} \label{alg_closed_higher_rank}
 Let $f\colon  Y \to X$ be an étale morphism of smooth curves over an algebraically closed analytic field $(k,k^+)$ of finite rank.
 Then the wild degree is constant on sufficiently small constructible segments.
\end{proposition}


\begin{proof}
 Let~$p$ be the residue characteristic of $(k,k^+)$.
 We assume that it is positive because otherwise there is nothing to prove.
 We start with a point $y \in Y$ where~$f$ is wild and a segment~$I$ at~$y$.
 By \cref{red_fet} there is a diagram
 \[
  \begin{tikzcd}
   Y'	\ar[r,"\text{tame}"]	\ar[d,"\text{finite étale}"',"f'"]	& Y	\ar[d,"f"]	\\
   X'	\ar[r,"\text{tame}"']									& X
  \end{tikzcd}
 \]
 such that
 \begin{itemize}
  \item $y$ lies in the image of $Y' \to Y$,
  \item $f'\colon Y' \to X'$ is a composition of (finitely many) finite étale morphisms of degree~$p$.
 \end{itemize}
 Choose a preimage $y' \in Y'$ of $y$.
 By \cref{segment_homeomorphism}~(i), the preimage~$I'$ of~$I$ in~$Y'$ is a constructible segment at~$y'$.
 Shrinking~$I'$ we may assume that $I' \to I \to f(I)$ is a homeomorphism onto its image (see \cref{segment_homeomorphism}~(ii)).
 Suppose that we know that the wild degree is constant on sufficiently small constructible segments for morphisms of degree~$p$.
 This is also true for tame morphisms as for such morphisms the wild degree is constant of value~$1$.
 Then the claim follows by repeatedly applying \cref{sandwich}.
 Therefore, we have reduced ourselves to treating finite étale morphisms of degree~$p$.

 We consider the segment~$I_\eta$ in the generic fiber~$Y_\eta$ defined as the preimage of the generic point $\eta \in S^\times$ under the morphism $I \to S^\times$.
 It contains the minimal generalization~$y_\eta$ of~$y$ in~$Y_\eta$.
 If~$f$ is wild at~$y_\eta$, we can use \cref{wild_all_intervals} to construct a subsegment $I'_\eta$ of~$I_\eta$ at~$y_\eta$ where~$f$ is wild.
 We now choose a constructible subsegment~$I'$ of~$I$ whose generic fiber is~$I'_\eta$ using \cref{lift_segment} below.
 Then~$f$ is wild on~$I'$.
 Since~$f$ is finite of degree~$p$ this means that the wild degree is constant on~$I'$.

 Let us now assume that~$f$ is not wild at~$y_\eta$.
 Then~$y_\eta$ has to be divisorial because otherwise the specialization field of~$y_\eta$ would equal $k^\circ/k^{\circ\circ}$, which is algebraicaly closed, and it would not be possible for~$f$ to be wild at~$y$.
 We consider the topological closure~$\overline{\{y_\eta\}}$ of the point~$y_\eta$ in~$Y$.
 It contains some infinitesimal points in~$Y_\eta$ and otherwise consists of points in other fibers, among them~$y$.
 By \cite[Corollary~3.5.3]{adiccurves}, the subspace
 \[
  Y_{y_\eta}^\succ := \overline{\{y_\eta\}} \setminus Y_\eta
 \]
 is naturally homeomorphic to a smooth curve over $(k^\circ/k^{\circ\circ},k^+/k^{\circ\circ})$ equipped with the valuation topology.
 More precisely, it is an open subspace of the analytification of the unique smooth curve over $k(\eta)^\succ = k^\circ/k^{\circ\circ}$ with function field $k(y_\eta)^\succ$.
 We now do the same for $x = f(y)$ and $x_\eta = f(y)_\eta = f(y_\eta)$.


 In total we get a morphism
 \[
  f_{y_\eta}^\succ \colon Y_{y_\eta}^\succ \longrightarrow X_{x_\eta}^\succ
 \]
 of smooth curves over $(k^\circ/k^{\circ\circ},k^+/k^{\circ\circ})$.
 Since $f$ is unramified at~$y_\eta$, the extension $k(y_\eta)^\succ/k(x_\eta)^\succ$ of specialization fields is finite separable of degree~$p$.
 Therefore, the corresponding morphism of curves is generically étale.
 We can view the point~$y$ as an element of $Y_{y_\eta}^\succ$.
 By assumption,~$f_{y_\eta}^\succ$ is wild at~$y$.

 Applying \cref{lift_segment} to the subsegment $\{y_\eta\} \subseteq I_\eta$, we can shrink~$I$ such that it is contained in~$Y_{y_\eta}^\succ$.
 In order to turn~$I$ into a segment of~$Y_{y_\eta}^\succ$, we have to modify its structure sheaves~$\cP_I$ and~$\cP_I^+$.
 We want to switch from~$\lvert k^\times \rvert$ as a base to the maximal convex subgroup $\Delta \subseteq \lvert k^\times \rvert$ (compare with the description of the fractal structure in \cite[\S~3.5]{adiccurves}).
 If $I = [a,b]_s$, we know that $I_\eta = \{y_\eta\}$ and we can normalize~$a$ and~$b$ in such a way that~$y_\eta$ corresponds to the ranger~$1$.
 Then $a$ and~$b$ are contained in the maximal proper convex subgroup~$\Delta$ of $\lvert k^\times \rvert$ and $I^\succ = \Spa(\Delta[a/t,t/b])$ is homeomorphic to~$I$ but its structure sheaves are the restrictions of the piecewise monomial functions of~$Y_{y_\eta}^\succ$.
 In particular,~$I^\succ$ is a segment at~$y$ in $Y_{y_\eta}^\succ$.
 The cases $I = [a,b)_s$, $I = (a,b]_s$, $I = (a,b)$, and $I = \{y\} \cup \bigcup_{i \in I} [a_i,b]$ are treated in the same way.

 The rank of the base field $(k^\circ/k^{\circ\circ},k^+/k^{\circ\circ})$ is one less than the rank of $(k,k^+)$.
 Consequently, we can assume by induction that the theorem holds over $(k^\circ/k^{\circ\circ},k^+/k^{\circ\circ})$.
 This means that we can find a constructible subsegment $I'^\succ \subseteq I^\succ$ at~$y$ where the wild degree of~$f_{y_\eta}^\succ$ is constant.

 The segment~$I'^\succ$ corresponds to a subsegment $I' \subseteq I$ of valuations that are compositions of~$y_\eta$ with a valuation of $k(y_\eta)^\succ$.
 The fact that~$f_{y_\eta}^\succ$ is wild at these points translates to the respective valuation of $k(y_\eta)^\succ$ being wild in $k(y_\eta)^\succ/k(x_\eta)^\succ$.
 Composing them with~$y_\eta$ we obtain valuations that are wild in $k(y_\eta)/k(x_\eta)$.
 In other words this means that~$f$ is wild on~$I'$.
\end{proof}






\begin{lemma} \label{lift_segment}
Let~$X$ be a smooth curve over~$S$ and let $x_1 \succ x_2$ be a specialization of points in~$X$ lying over $s_1 \succ s_2$ in~$S$.
 Suppose $I_2 \subseteq X^\times$ is a segment at~$x_2$ and $I_1$ a segment at~$x_1$ contained in $(I_2)_{s_1}$.
 Then there is a subsegment $I \subseteq I_2$ at~$x_2$ such that $I_{s_1} = I_1$
\end{lemma}

\begin{proof}
Without loss of generality we may assume~$I_2$ is either isomorphic to $[a,b]_{s_2}$, $[a,b)_{s_2}$, $(a,b]_{s_2}$, or $[a,b)_{s_2}$ for~$a$ divisorial or~$0$ and~$b$ divisorial or to $\{x\} \cup \bigcup_{i \in I} [a_i,b]_s$ for~$x_2$ unibranch and~$a_i$ and~$b$ divisorial.
 For $I_2 = [a,b]_{s_2}$ with divisorial rangers~$a$ and~$b$ we get $(I_2)_{s_1} = [a_{s_1},b_{s_1}]_{s_1}$, where~$a_{s_1}$ and~$b_{s_1}$ are the images of~$a$ and~$b$ under the projection $\cR_{\lvert k(s_2)^\times\rvert} \to \cR_{\lvert k(s_1)\times\rvert}$.
 Then $I_1 \subseteq (I_2)_{s_1}$ is a subinterval with divisorial endpoints~$a'_1$ and~$b'_1$ such that $a_{s_1} \le a'_1 \le x_1 \le b'_1 \le b_{s_1}$.
 We obtain~$I$ by choosing lifts $a'$ and~$b'$ of $a'_1$ and~$b'_1$ in~$\cR_{\lvert k(s_2)^\times \rvert}$ satisfying $a \le a' \le x_2 \le b' \le b$ (where we identified~$x_2$ with the corresponding ranger in $[a,b]_{s_2}$).
 The other cases are similar.
\end{proof}

Using the following lemma we can pass from algebraically closed to separably closed base fields.

\begin{lemma} \label{purely_inseparable_wild}
 Let $f\colon  Y \to X$ be an étale morphism of smooth curves over an affinoid field $(k,k^+)$ of finite rank.
 We consider the base change of~$f$ along a finite purely inseparable extension $(\ell,\ell^+)/(k,k^+)$:
 \[
  f_\ell \colon  Y_\ell \longrightarrow X_\ell
 \]
 Then for any $y_\ell \in Y_\ell$ mapping to $y \in Y$ we have the equality
 \[
  \wild_{f_\ell}(y_\ell) = \wild_f(y)
 \]
 of wild degrees.
\end{lemma}

\begin{proof}
 We set $x_\ell := f_\ell(y_\ell)$ and $x:=f(y)$ and examine the completed residue field extensions:
 \[
  \begin{tikzcd}[row sep=small]
   \widehat{k(y_\ell)}	\ar[rd,-]	\ar[dd,-]								\\
   										& \widehat{k(x_\ell)}	\ar[dd,-]	\\
   \widehat{k(y)}			\ar[rd,-]											\\
   										& \widehat{k(x)}.
  \end{tikzcd}
 \]
 The slanted horizontal extensions are separable and the vertical extensions are purely inseparable.
 It follows that the vertical extensions are totally ramified of the same degree.
 Since the wild degree is multiplicative in composite extensions, the wild degree of the two slanted horizontal extensions needs to be the same.
 This is a reformulation of the statement
 \[
  \wild_{f_\ell}(y_\ell) = \wild_f(y)
 \]
\end{proof}

\begin{corollary} \label{wild_degree_constant_separably_closed}
 Let $f\colon  Y \to X$ be an étale morphism of smooth curves over a separably closed affinoid field $(k,k^+)$ of finite rank.
 Then the wild degree $\wild_f$ is constant on sufficiently small constructible segments of~$Y$.
\end{corollary}

\begin{proof}
 Taking the base change along a purely inseparable extension of $(k,k^+)$ does not change the underlying topological spaces of~$X$ and~$Y$.
 In particular, the base change preserves constructible segments.
 Therefore, the corollary follows by combing \cref{alg_closed_higher_rank} with \cref{purely_inseparable_wild}.
\end{proof}

\begin{corollary} \label{basechange_to_bar_k_stable}
 Let $(k,k^+)$ be an affinoid field with separable closure $(\bar{k},\bar{k}^+)$.
 We consider an étale morphism
 \[
  f \colon  Y \longrightarrow X
 \]
 of smooth curves over $(k,k^+)$ and its basechange
 \[
  \bar{f} \colon  Y_{\bar{k}} \longrightarrow X_{\bar{k}}
 \]
 to $(\bar{k},\bar{k}^+)$.
 For a point $\bar{y} \in Y_{\bar{k}}$ with image $y \in Y$ we assume that
 \[
  \wild_f(y) = \wild_{\bar{f}}(\bar{y}).
 \]
 Then the wild degree of~$f$ is constant on all sufficiently small segments at~$y$.
\end{corollary}

\begin{proof}
 By upper semicontinuity of the wild degree (\cref{wd_upper_semicont}) we have on every sufficiently small segment~$I$ at~$y$
 \begin{equation} \label{semicontinuity_interval_downstairs}
  \wild_f(y') \le \wild_f(y) \quad \forall y' \in I.
 \end{equation}
 On the other hand, on every sufficiently small segment~$\bar{I}$ at~$\bar{y}$ we have the equality
 \[
  \wild_{\bar{f}}(\bar{y}') = \wild_{\bar{f}}(\bar{y}) = \wild_f(y)	\quad \forall \bar{y}' \in \bar{I}.
 \]
 (see \cref{wild_degree_constant_separably_closed}).
 The image of such a segment~$\bar{I}$ in~$Y$ is a segment at~$y$ and if~$\bar{I}$ is small enough, its image~$I$ satisfies the inequalities~(\ref{semicontinuity_interval_downstairs}).
 In addition, for every point $\bar{y}' \in \bar{I}$ with image $y' \in I$ we have the inequality
 \[
  \wild_f(y') \ge \wild_{\bar{f}}(\bar{y}').
 \]
 This forces the inequality in~(\ref{semicontinuity_interval_downstairs}) to be an equality.
\end{proof}

\section{The wild degree for curves over a general field} \label{sect:wild_degree_general}

We continue studying the wild degree of an étale morphism $f\colon  Y \to X$ of curves over an analytic field $(k,k^+)$ of finite rank.
Roughly there are two key cases we need to understand in order to get the full picture.
The first one treats curves over algebraically closed fields and has been studied in \cref{sect:alg_closed_ramification_locus}.
The second case deals with étale morphisms on the affine line of the form
\[
 \pi \colon  \A_\ell^1 \longrightarrow \A_k^1
\]
obtained by base change along a finite separable extension $\ell/k$.
It will turn out to be handy not to completely separate these two cases, especially when studying the wild degree on unibranch points.
Consequently, this section is primarily about the affine line but we will use the geometric case in order to understand~$\pi$.
In the end we combine the results of this and the previous section to show that the wild degree is constant on sufficiently small constructible segments for any étale morphism of smooth curves.

\subsection{The wild degree on the affine line} \label{sect:affine_line}

In this subsection we study the finite étale morphism
\[
 \pi \colon  \A_\ell^1 \longrightarrow \A_k^1
\]
induced by a finite separable extension $(\ell,\ell^+)/(k,k^+)$ of analytic fields.
Our aim is to understand the ramification behaviour along the segment $[0,\infty)$ connecting the two classical points~$0$ and~$\infty$ in $\P_\ell^1$.
The special fiber of the segment $[0,\infty)$ is parametrized by the set of bounded rangers~$\cR^b_{\lvert \ell^\times \rvert}$ together with~$0$.
Explicitly, an element $y \in [0,\infty)$ in the special fiber is represented by a valuation of $\ell[T]$ of the form
\[
 y(\sum_i a_i T^i) = \max_i \lvert a_i \rvert \gamma^i,
\]
where~$\gamma$ is either~$0$ or an element of a totally ordered group~$\Gamma$ containing~$\lvert \ell^\times \rvert$.
The pair $(\Gamma,\gamma)$ is a bounded ranger and its equivalence class $r_\gamma \in \cR_{\lvert \ell^\times \rvert}$ is called the \emph{radius} of~$y$.
We also denote it by~$r_y$.

If we start with a radius $r \in \cR^b_{\lvert \ell^\times \rvert} \cup \{0\}$, we denote the corresponding point on the segment $[0,\infty)$ by~$y_r$.
For $r > 0$, it corresponds to a valuation ring of $\ell(T)$, for which we write $\ell(T)_r^+$.
The completed residue field extension of~$y_r$ is identified with the completion of
\[
 (\ell(X),\ell(X)_r^+) / (\ell,\ell^+).
\]
The same reasoning applies to~$k$ instead of~$\ell$ and we use the same notation $(k(X),k(X)_r^+$) for the affinoid field corresponding to the ranger~$r$.
Since $\cR_{\lvert \ell^\times \rvert} = \cR_{\lvert k^\times \rvert}$, the segment $[0,\infty)$ in $\A_\ell^1$ is naturally homeomorphic to the segment $[0,\infty)$ in $\A_k^1$.
We tend to use~$x_r$ (as opposed to~$y_r$) when we refer to the point in $[0,\infty) \subseteq \A_k^1$ corresponding to~$r$.

In order to understand the wild degree of~$\pi$ on $[0,\infty)$, we need to study the wild degree of the completed residue field extensions
\[
 (k(y_r),k(y_r)^+)^\wedge/(k(x_r),k(x_r)^+)^\wedge
\]
for $r \in \cR^b_{\lvert k^\times \rvert} \cup \{0\}$.
The case $r=0$ is easy as $k(x_0) = k$ and $k(y_0) = \ell$.
For $r > 0$ the residue field extension is the completion of
\[
 (\ell(T),\ell(T)_r^+)/(k(T),k(T)_r^+).
\]
Since~$\ell/k$ is separable, the wild degree of the above extension is the same as the wild degree of its completion.
So in a nutshell we need to study the wild degree of the extension $(\ell(T),\ell(T)_r^+)/(k(T),k(T)_r^+)$.

%
%

\begin{lemma} \label{wild_monomial_valuation}
 Let $(\ell,\ell^+)/(k,k^+)$ be a finite separable extension of affinoid fields and $r \in \cR_{\lvert k^\times \rvert} \setminus \{\infty\}$.
 Then
 \[
  \wild\left((\ell(T),\ell(T)_r^+)/(k(T),k(T)_r^+)\right) = \wild\left((\ell,\ell^+)/(k,k^+)\right).
 \]
\end{lemma}

\begin{proof}
Without loss of generality we may assume that $(k,k^+)$ is tame closed and complete. In this case the wild degree of the extension $\ell/k$ equals $[\ell:k]$. Furthermore, we can decompose $\ell/k$ into finitely many extensions $\ell_{i+1}/\ell_i$ of degree~$p$. What we need to show is that the induced extension $(\ell_{i+1}(T),\ell_{i+1}(T)_r^+)/(\ell_i(T),\ell_i(T)_r^+)$ is wild. We are thus reduced to the case where $\ell/k$ is a wild extension of degree~$p$ and we want to show that $\ell(T)/k(T)$ is wild.

We will use the following two simple facts for an element $\alp\in \ell\setminus k$:

1. {\it Orthogonality of the monomial valuation on $k(t)$: the rangers $d_{\alp/k}=\inf_{c\in k}|\alp-c|$ and $d_{\alp/k(t)}=\inf_{c\in k(t)}|\alp-c|$ are equal.}
Indeed, assume to the contrary that $d_{\alp/k}>d_{\alp/k(t)}$ and choose $h\in k(t)$ with $|\alp-h|<d_{\alp/k}$ in $\ell(t)$.
Find a presentation $h=f/g$ with $f=\sum_{i=0}^m f_it^i$ and $g=\sum_{i=0}^n g_it^i$, and fix $i$ such that $|g|=|g_it^i|$.
Then $$|\alp-f_i/g_i|\cdot |g_it^i|=|(\alp g_i-f_i)t^i|\le|\alp g-f|<|g|d_{\alp/k}=|g_it^i|d_{\alp/k}$$ and hence $|\alp-f_i/g_i|<d_{\alp/k}$, which is absurd.

2. {\it Orthogonality for a valued extension $\ell/k$ of degree $p$: the extension is defectless if and only if the ranger $d_{\alp/k}$ is principal (the infimum is attained).
Furthermore, assume that $\ell/k$ is defectless and $\alp$ is $k$-orthogonal (that is, $d_{\alp/k}=|\alp|$).
Then $\ell/k$ is wild if and only if $\alp^p$ is not $k$-orthogonal.}
This claim is straightforwardly checked by studying the cases when $e_{\ell/k}=p$ and $f_{\ell/k}=p$, and we skip the details.
It seems a bit artificial, but allows not to distinguish these cases further and to also completely ignore the question what is $r$, while otherwise we would have to separately study a few cases, including the case when $|\alp|,r\notin|k^\times|$ and $|\alp|^p,r^p\in|k^\times|$, and hence $e_{\ell/k}=p=f_{\ell(t)/k(t)}$.

Now we are ready to prove that $\ell(t)/k(t)$ is wild. By claim 1 we know that $d_{\alp/k}=d_{\alp/k(t)}$ (and the same for $\alp^p$), hence either both rangers are non-principal and then both $\ell/k$ and $\ell(t)/k(t)$ are defect wild extensions by claim 2, or we can translate $\alp$ by an element of $k$ making it $k$-orthogonal.
In this case, $\alp^p$ is not $k$-orthogonal by claim 2, hence it is also not $k(t)$-orthogonal and the extension $\ell(t)/k(t)$ is wild.

\end{proof}

\begin{proposition} \label{wild_affine_line}
 Let $(\ell,\ell^+)/(k,k^+)$ be a finite separable extension of affinoid fields.
 We write
 \[
  \pi \colon  \A_\ell^1 \longrightarrow \A_k^1.
 \]
 for the natural projection.
 Then the wild degree $\wild_{\pi}$ is constant on the special fiber of $[0,\infty) \subseteq \A_\ell^1$ of value $\wild(\ell/k)$.
\end{proposition}

\begin{proof}
 As we have discussed at the beginning of the subsection, the points on the special fiber of $[0,\infty)$ are parametrized by $\cR^b_{\lvert k^\times \rvert} \cup \{0\}$.
 For a given radius $r \in \cR^b_{\lvert k^\times \rvert} \cup \{0\}$ the wild degree at the corresponding point $y_r \in [0,\infty)$ is the wild degree of $(k(y_r),k(y_r)^+)/(k(x_r),k(x_r)^+)$.
 This extension equals $(\ell,\ell^+)/(k,k^+)$ in case $r=0$ and otherwise its completion is identified with the completion of $(\ell(T),\ell(T)_r^+)/(k(T),k(T)_r^+)$.
 Therefore, the result follows from \cref{wild_monomial_valuation}.
\end{proof}

\subsection{Wild degree for non-endpoints} \label{sect:non-endpoints}

Using the results from \cref{sect:affine_line} we can now treat general curves.
However, we can not yet say anything about the behaviour of the wild degree at endpoints.

\begin{proposition} \label{etale_morphism_to_A1}
 Let~$X$ be a smooth curve over an affinoid field $(k,k^+)$ and $x \in X$ a point in the special fiber.
 Then there is an open neighborhood $U \subseteq X$ of~$x$ and an étale morphism $\varphi\colon U \to \A_k^1$.

 If~$x$ is not an endpoint, $\varphi$ can be arranged in such a way that $\varphi(x) \subseteq [0,\infty)$.
\end{proposition}

\begin{proof}
 Remember from \cite[\S~2.1]{CTT16} that a parameter of $k(x)$ is an element $t \in k(x)$ such that $\widehat{k(x)}/\widehat{k(t)}$ is finite separable.
 Choosing a parameter $t \in k(x)$ defines an étale morphism of a neighborhood of~$x$ to $\A^1_k$.
 Let $x_0 \in \A^1_k$ be the image of~$x$.
 If we want~$x_0$ to lie on $[0,\infty)$, $t$ has to be a monomial parameter, i.e. $\lvert t(x)\rvert$ has to be minimal among all values $\lvert (t-c)(x) \rvert$ for $c \in \bar{k}$.
 This has to be ensured in case $x$ is not an endpoint meaning that $x$ is an edge or divisorial point.
 If~$x$ is an edge point, we choose~$t$ such that $\lvert t(x) \rvert$ is not contained in $\lvert k^\times \rvert$.
 In case~$x$ is divisorial we take any~$t$ of absolute value~$1$ such that $k(x)^\succ/k(t)^\succ$ is separable algebraic.
 In both cases~$t$ is automatically monomial and thus induces a morphism to $\A^1_k$ that maps~$x$ to $[0,\infty)$.
\end{proof}

\begin{proposition} \label{wild_degree_nonendpoints}
 Let $f \colon  Y \to X$ be an étale morphism of smooth curves over an affinoid field $(k,k^+)$.
 For every point $y \in Y$ that is not an endpoint the wild degree $\wild_f$ is constant on every sufficiently small constructible segment~$I$ at~$y$.
\end{proposition}

\begin{proof}
 Let~$x$ be the image of~$y$ in~$X$.
 Using \cref{etale_morphism_to_A1} we construct an étale morphism
 \[
  \varphi \colon  U \to \A_k^1
 \]
 from an open neighborhood~$U$ of~$x$ to the affine line mapping~$x$ to a point on the segment $[0,\infty)$.
 Without loss of generality we can replace~$X$ by~$U$ and henceforth assume that~$\varphi$ is defined on all of~$X$.

 We consider the basechanges
 \[
  Y_{\bar{k}} \overset{\bar{f}}{\longrightarrow} X_{\bar{k}} \overset{\bar{\varphi}}{\longrightarrow} \A^1_{\bar{k}}
 \]
 of~$f$ and~$\varphi$ to an algebraic closure~$\bar{k}$ of~$k$ and a point $\bar{y} \in Y_{\bar{k}}$ above~$y$.
 Then $\wild_{\bar{f}}(\bar{y})$ is possibly smaller than $\wild_f(y)$ and similarly for $\wild_{\varphi}$.
 It can be deduced from continuity of the tame topology (\cite[Proposition~5.3]{HueAd}) that there is a finite extension $(\ell,\ell^+)/(k,k^+)$ such that the base changes
 \[
  Y_\ell \overset{f_\ell}{\longrightarrow} X_\ell \overset{\varphi_\ell}{\longrightarrow} \A^1_\ell
 \]
 have wild degree $\wild_{\bar{f}}(\bar{y})$ at the image~$y_\ell$ of~$\bar{y}$ and $\wild_{\bar{\varphi}}(\bar{x})$ at the image~$x_\ell$ of~$\bar{y}$, respectively.
 Then by \cref{basechange_to_bar_k_stable} the wild degrees of~$f_\ell$ and~$\varphi_\ell$ are constant on every sufficiently small segment at~$y_\ell$ and~$x_\ell$, respectively.

 Combining the previous two paragraphs of the proof we obtain a diagram of étale morphisms
  \begin{equation} \label{diagram_Y_X_A}
  \begin{tikzcd}
   Y_\ell	\ar[r,"f_\ell"]	\ar[d,"\pi_X"]	& X_\ell	\ar[r,"\varphi_\ell"]	\ar[d,"\pi_X"]		& \A_\ell^1	\ar[d,"\pi_\A"]	\\
   Y		\ar[r,"f"]						& X			\ar[r,"\varphi"]							& \A_k^1.	
  \end{tikzcd}
 \end{equation}
 Now we start with a constructible segment $I \subseteq Y$ at~$y$.
 Using \cref{segment_homeomorphism} we choose a constructible segment $I_\ell \subseteq Y_\ell$ at~$y_\ell$ mapping into~$I$ such that
 \[
  I_\ell \longrightarrow \pi_\A \circ \varphi_\ell \circ f_\ell(I_\ell)
 \]
 is a homeomorphism.

 Then the map from~$I_\ell$ onto its image in any of the spaces appearing in diagram~(\ref{diagram_Y_X_A}) is a homeomorphism.
 Shrinking~$I_\ell$ further we may assume that the wild degrees
 \[
  \wild_{f_\ell},~\wild_{\varphi_\ell},~\wild_{\pi_\A}
 \]
 are constant on~$I_\ell$, respectively on its images in~$X_\ell$ and~$\A^1_\ell$ (see \cref{basechange_to_bar_k_stable} and \cref{wild_affine_line}).
 By \cref{sandwich} the wild degree
 \[
  \wild_{\pi_\A \circ \varphi_\ell \circ f_\ell} = \wild_{\varphi \circ f \circ \circ \pi_Y}
 \]
 is constant on~$I_\ell$.

 Now the same \cref{sandwich} implies that the wild degree $\wild_f$ is constant on the image of~$I_\ell$ in~$Y$.
\end{proof}

\subsection{The affine line revisited}

Let $(\ell,\ell^+)/(k,k^+)$ be a finite separable extension of affinoid fields.
We come back to the study of the base extension morphism
\[
 \pi_\ell \colon  \A^1_\ell \longrightarrow \A^1_k
\]
on the affine line.
We know from \cref{wild_degree_nonendpoints} that the wild degree is well behaved at non-endpoints.
In this section we explain how to treat the endpoints.

For an element $a \in k$ we consider the automorphism
\[
 t_a \colon  \A^1_k \longrightarrow \A^1_k
\]
that is given by sending the variable~$T$ to $T-a$.
On a classical point~$x$ defined by setting $T=b$ in $\bar{k}$ it acts by sending~$b$ to $b+a$.
So this automorphism~$t_a$ is a translation by~$a$.
For a general point $x \in \A^1_k$ we often write $x + a$ when we mean $t_a(x)$.

\begin{lemma} \label{translation_by_a}
 Let $x \in \A^1_k$ be a point of radius $r \in \cR^b_{|k^\times|} \cup \{0\}$ and $a \in k$.
 Then $x = x + a$ if and only if $|a| \le r$.
\end{lemma}

\begin{proof}
 Let $(\bar{k},\bar{k}^+)$ be an algebraic closure of $(k,k^+)$.
 Suppose that $|a| > r$.
 The radius of~$x$ being less than~$|a|$ implies that there is $c \in \bar{k}$ such that
 \[
  |(T-c)(x)| < |a|.
 \]
 For the image point $t_a(x)$ we get
 \[
  |(T-c)(x+a)| = |(T-a-c)(x)| = |((T-c)-a)(x)| = |a|,
 \]
 whence $x+a \ne x$.

 Now let~$x$ be a $T_c$-monomial point (with respect to the variable $T_c := T -c$ for some $c \in \bar{k}$) of radius greater than~$|a|$.
 Then $x+a$ is the following valuation:
 \[
  |(\sum_i b_i (T-c)^i)(x+a)| = |(\sum_i b_i (T-a-c)^i)(x)| = |(\sum_i b_i (T-c)^i)(x)|.
 \]
 So $x+a = x$.
 Every point of radius $r \le |a|$ can be approximated by monomial points of radius greater than~$|a|$.
 Hence the result follows by continuity.
\end{proof}

\begin{lemma} \label{wild_degree_A_1_infinite}
 Let $(\ell,\ell^+)/(k,k^+)$ be a finite separable extension of affinoid fields.
 Suppose that the value group $|k^\times|$ is non-discrete and the residue field~$\tilde{k}$ infinite.
 We consider the projection
 \[
  \pi_\ell \colon  \A^1_\ell \longrightarrow \A^1_k
 \]
 and a point $x \in \A^1_\ell$.
 Then the wild degree $\wild_{\pi_\ell}$ is constant on every sufficiently small constructible segment at~$x$.
\end{lemma}

\begin{proof}
 By \cref{wild_degree_nonendpoints} we know the statement for non-endpoints.
 Suppose it is not true for some endpoint $x$.
 By semicontinuity we find an open neighborhood~$U$ of~$x$ such that
 \[
  \wild_{\pi_\ell}|_U \le \wild_{\pi_\ell}(x).
 \]
 In particular, the wild degree is at most $\wild_{\pi_\ell}(x)$ on every sufficiently small constructible segment~$I$ at~$x$.
 Since~$x$ is an endpoint, every constructible segment at~$x$ contains monomial points of radius greater than the radius~$r_x$ of~$x$ that correspond to discs containing~$x$.
 In this way we find a monomial point $y \in \A^1_\ell$ of radius~$r_y > r_x$ corresponding to a disc containing~$x$ such that
 \[
  \wild_{\pi_\ell}(y) < \wild_{\pi_\ell}(x).
 \]
 As~$|k^\times|$ is non-discrete, we can modify~$y$ to achieve that $r_y \in |k^\times|$.
 For an element $c \in k$ of valuation $|c| = r_y$ we consider the translation~$t_c$ by~$c$ on $\A^1_\ell$ and $\A^1_k$:
 \[
  \begin{tikzcd}
   \A^1_\ell	\ar[r,"t_c"]	\ar[d,"\pi_\ell"]	& \A^1_\ell	\ar[d,"\pi_\ell"]	\\
   \A^1_k		\ar[r,"t_c"]					& \A^1_k.
  \end{tikzcd}
 \]
 By \cref{translation_by_a} the point $t_c(x)$ is different from~$x$.
 But the wild degree $\wild_{\pi_\ell}$ is the same at~$x$ and $x+c$ since~$t_c$ is defined over~$k$.
 By our assumption on the cardinality of~$\tilde{k}$ there are infinitely many elements~$c' \in k$ of valuation $|c'| = r_y$ such that the residue classes of $c'/c$ in $\tilde{k}$ are pairwise different.
 According to \cref{translation_by_a} this implies that the points $x + c'$ for $c'$ as above are pairwise different.
 More precisely, for two such elements $c',c'' \in k$ we have
 \[
  d(x+c',x+c'') = |c-c'| = r_y.
 \]
 \cref{translation_by_a} also implies that all the points $x+c'$ are contained in the disc corresponding to~$y$.
 However, no two points $x+c'$ and $x+c''$ are contained in a disc of smaller radius.
 A basis of open neighborhoods for~$y$ is given by the sets obtained by removing from the disc corresponding to~$y$ finitely many discs of smaller radius (compare \cite[Theorem~4.6.5]{adiccurves}).

 By the above considerations the points $x+c'$ are not all contained in a finite union of discs of radius less than~$r_y$.
 Therefore, every open neighborhood of~$y$ contains one of the points $x+c'$.
 But then
 \[
  \wild_{\pi_\ell}(x+c') > \wild_{\pi_\ell}(y)
 \]
 contradicts upper semicontinuity of the wild degree at~$y$.
\end{proof}

\begin{proposition} \label{wild_degree_A_1}
 Let $(\ell,\ell^+)/(k,k^+)$ be a finite separable extension of affinoid fields.
 We consider the projection
 \[
  \pi_\ell \colon  \A^1_\ell \longrightarrow \A^1_k
 \]
 and a point $x \in \A^1_\ell$.
 Then the wild degree $\wild_{\pi_\ell}$ is constant on every sufficiently small constructible segment at~$x$.
\end{proposition}

\begin{proof}
 We can find a tame extension $(k',k'^+)/(k,k^+)$ such that~$k'^\succ$ is infinite and the value group $|k'^\times|$ is not discrete.
 Setting $\ell' = k'\ell$ (in some algebraic closure) we obtain a diagram
 \[
  \begin{tikzcd}
   \A^1_{\ell'}	\ar[r]	\ar[d,"\pi_{\ell'}"]	& \A^1_\ell	\ar[d,"\pi_\ell"]	\\
   \A^1_{k'}	\ar[r]							& \A^1_k.
  \end{tikzcd}
 \]
 Let~$I$ be a constructible segment at~$x$.
 It is contained in some ray of the form $[x_0,\infty)$.
 Choosing a preimage $x'_0 \in \A^1_{\ell'}$ of~$x_0$ we obtain a ray $[x'_0,\infty) \subseteq \A^1_{\ell'}$ mapping homeomorphically to $[x_0,\infty)$.
 Let $I' \subseteq [x'_0,\infty)$ be the preimage of~$I$.

 For every point $y' \in \A^1_{\ell'}$ mapping to $y \in \A^1_\ell$ we have an equality of wild degrees
 \[
  \wild_{\pi_{\ell'}}(y') = \wild_{\pi_\ell}(y).
 \]
 Using \cref{wild_degree_A_1_infinite} we find a constructible subsegment $J' \subseteq I'$ at~$x'$ where the wild degree $\wild_{\pi_{\ell'}}$ is constant.
 Then $\wild_{\pi_\ell}$ is constant on the image of~$J'$ in~$\A^1_\ell$.
\end{proof}

\subsection{The general case}

Having understood the wild degree on the affine line also for endpoints, we can repeat the reasoning in \cref{sect:non-endpoints} to treat arbitrary points on curves.

\begin{theorem} \label{wild_degree_constant_segments}
 Let $(k,k^+)$ be an affinoid field of finite rank and $f \colon  Y \to X$ an étale morphism of smooth curves over $(k,k^+)$.
 Then the wild degree $\wild_f$ is constant on sufficiently small constructible segments of~$Y$.
\end{theorem}

\begin{proof}
 We repeat the proof of \cref{wild_degree_nonendpoints} but with one deviation:
 Now we are given an arbitrary point $y \in Y$ which might be an endpoint.
 When constructing an étale morphism from an étale neighborhood~$U$ of~$y$ to $\A^1_k$, we cannot guarantee that~$y$ be mapped to the interval $[0,\infty)$.
 However, this does not bother us anymore as we now know by \cref{wild_degree_A_1} that the wild degree is constant on \emph{every} small enough constructible segment of~$\A^1_k$.
%
%
\end{proof}

Finally, let us eliminate the assumption on the rank.

\begin{corollary} \label{wild_locus_analytic_curve}
 Let $(k,k^+)$ be an analytic affinoid field of finite rank and $f \colon  Y \to X$ an étale morphism of smooth curves over $(k,k^+)$.
 Then the wild locus of~$f$ is either empty or contains a divisorial point.
\end{corollary}

\begin{proof}
 Suppose the wild locus~$R_f$ contains a point $y \in Y$.
 By \cref{wild_degree_constant_segments} the wild degree $\wild_f$ is constant on some constructible segment $I$ at~$y$.
 Then $I \subseteq R_f$.
 Since every constructible segment contains a divisorial point by \cref{segment_divisorial}, we are done.
\end{proof}

\section{Analytification} \label{sect:analytification}

Up until now we have treated curves over an \emph{analytic} affinoid field.
However, for some applications we are also interested in curves over a discrete affinoid field.
In order to connect to what we have already achieved, we use analytification in the following way.
Let $(k,k^+)$ be an analytic affinoid field of finite rank with a pseudouniformizer~$\varpi$.
We denote by $(k,k^+)^{\disc}$ the same affinoid field but with discrete topology.
If $X/(k,k^+)^{\disc}$ is an adic space locally of finite type, we define its analytification
\[
 X^{\an} := X \times_{(k,k^+)^{\disc}} (k,k^+).
\]
By \cite[Corollary~1.2.3~i)]{Hu96}, $X^{\an}$ is locally of finite type over~$(k,k^+)$.
Moreover, by \cite[Proposition~1.4.7~iv)]{Hu96}, it is smooth if~$X$ is smooth.

\begin{lemma} \label{image_anlytification}
 The morphism $X^{\an} \to X$ is injective.
 Its image is given by all points $x \in X$ such that $|\varpi(x)|$ is cofinal in~$\lvert k(x)^\times \rvert$.
\end{lemma}

\begin{proof}
 We may assume that $X = \Spa(A,A^+)$ is affinoid.
 If $A = k[T_1,\ldots,T_m]/I$, then $X^{\an}$ is glued from an ascending chain of affinoids of the form $\Spa(B_n,B_n^+)$ with
 \[
  B_n = k\langle \varpi^n T_1,\ldots,\varpi^n T_m \rangle/I,
 \]
 where $\varpi \in k$ is a pseudouniformizer.
 Then
 \[
  \Spa(B_n,B_n^+) \to \Spa(A,A^+)
 \]
 is injective, as $A \to B_n$ has dense image.

 It is clear that for all points $x \in X^{\an}$ we have that $|\varpi(x)|$ is cofinal in~$\lvert k(x)^\times \rvert$.
 For the converse direction take $x \in X$ of rank greater or equal to one with $|\varpi(x)|$ cofinal in~$\lvert k(x)^\times \rvert$.
 Hence, for $n$ big enough, $|(\varpi^n T_i)(x)| \le 1$ for $i =1,\ldots,m$.
 Together with the cofinality of $|\varpi(x)|$ in~$\lvert k(x)^\times \rvert$, this shows that $x \in \Spa(B_n,B_n^+)$.
\end{proof}

Let us now specialize to the case of a curve~$C$ over $(k,k^+)^{\disc}$.
By \cref{image_anlytification} we know that we can identify the points~$x$ of~$C$ such that $\lvert \varpi(x) \rvert$ is cofinal in~$\lvert k(x)^\times \rvert$ with the points of the analytification~$C^\an$.
We call these points \emph{bounded}.
Let us now investigate the remaining points, the \emph{unbounded} ones.
We write~$\eta$ for the generic point of $\Spa(k,k^+)^{\disc}$.
It corresponds to the trivial valuation on~$k$.
For a point $x \in C$ we consider~$x_\eta$, its minimal (vertical) generalization in~$C_\eta$.
The curve~$C_\eta$ lies over the trivially valued field $(k,k)$.
Hence, by Ostrowski's theorem, its points are either horizontally divisorial or correspond to trivial valuations.
Explicitly, the point~$x_\eta$ is given by the valuation
\[
 \lvert \cdot (x_\eta)\rvert \colon k(x) \overset{\lvert \cdot (x) \rvert}{\longrightarrow } \lvert k(x)^\times \rvert \cup \{0\} \longrightarrow \left( \lvert k(x)^\times \rvert/\Delta_k \right) \cup \{0\},
\]
where~$\Delta_k$ is the convex subgroup of~$\lvert k(x)^\times \rvert$ generated by~$\lvert k^\times \rvert$.

\begin{lemma} \label{divisorial_an}
 Let~$C$ be a smooth curve over $(k,k^+)^{\disc}$.
 Then every point $x \in C$ is either trivial, contained in $C^{\an}$, or $x_\eta$ is horizontally divisorial.
\end{lemma}

\begin{proof}
 If $|\varpi(x)|$ is cofinal in~$\lvert k(x)^\times \rvert$, $x$ lies in the image of $C^{\an} \to C$ (see \cref{image_anlytification}).
 Now suppose that $|\varpi(x)|$ is not cofinal in~$\lvert k(x)^\times \rvert$.
 This happens precisely when either~$\lvert k(x)^\times \rvert$ is trivial or the convex subgroup of~$\lvert k(x)^\times \rvert$ generated by $\lvert \varpi(x) \rvert$ is non-trivial.
 Since~$\varpi$ is a pseudouniformizer of $(k,k^+)$, this convex subgroup equals the convex subgroup~$\Delta_k$ of~$\lvert k(x)^\times \rvert$ generated by~$\lvert k^\times \rvert$.
 In the first case~$x$ corresponds to a trivial valuation and in the second case the value group $\lvert k(x)^\times \rvert/\Delta_k$ of~$x_\eta$ is non-trivial.
 By Ostrowski's theorem,~$x_\eta$ has to be horizontally divisorial in the second case.
\end{proof}

\begin{proposition} \label{wild_locus_curves}
 Let $f\colon  Y \to X$ be an étale morphism of smooth curves over an affinoid field $(k,k^+)$ of finite rank.
 Then the wild locus $R_f$ is either empty or contains a divisorial point.
\end{proposition}

\begin{proof}
 If $(k,k^+)$ is analytic, the claim follows from \cref{wild_locus_analytic_curve}.
 Assume now that $(k,k^+)$ carries the discrete topology.
 Let $y \in Y$ be a point contained in~$R_f$.
 By \cref{divisorial_an} the point~$y$ is either trivial, bounded or $x_\eta$ is divisorial.
 Trivial valuations cannot be contained in~$R_f$, so this case is excluded.
 If~$y$ is bounded, it is contained in the wild locus of
 \[
  f^\an \colon  Y^\an \longrightarrow X^\an.
 \]
 Hence the result follows from the analytic case (\cref{wild_locus_analytic_curve}).

 Suppose in the sequel that $y_\eta$ is horizontally divisorial.
 We will reduce to the analytic case but will have to carry out a trick that involves adding additional points to our curves.
 Using \cref{red_fet} we may assume that $f$ is finite of $p$-power degree.
 For every point $y' \in Y \setminus R_f$ with image $x' \in X$, the residue field extension $k(y')/k(x')$ is unramified.
 We can replace~$Y$ and~$X$ by open neighborhoods of~$y$ and its image~$x$ in~$X$.
 In this way we achieve that the morphism~$f$ comes from a diagram of schemes
 \[
  \begin{tikzcd}
   C_Y		\ar[r,open] \ar[d,"\text{étale}"'] & C_Y^+		\ar[d]  \\
   C_X		\ar[r,open] \ar[d]					& C_X^+		\ar[d]  \\
   \Spec k	\ar[r,open]							& \Spec k^+.
  \end{tikzcd}
 \]
 where~$C_X^+$ and~$C_Y^+$ are normal relative curves over~$k^+$ with dense open subschemes~$C_Y$ and~$C_X$ contained in the generic fibre that are smooth curves over~$k$.
 Then~$f$ is the induced morphism
 \[
  f \colon  Y = \Spa(C_Y,C_Y^+) \longrightarrow X = \Spa(C_X,C_X^+).
 \]
 For instance,~$f$ takes this form if~$Y$ and~$X$ are affinoid.

 If $y_\eta \in R_f$, we are done.
 Otherwise we write~$y$ as a composition $y = y_\eta \circ w$, where~$w$ is a valuation on the specialization field $k(y_\eta)^\succ$.
 If~$y_\eta$ has a center~$c_y$ on~$C_Y$, the valuation~$w$ on $k(y_\eta)^\succ = k(c_y)$ defines a horizontal specialization of~$y$.
 It is bounded and contained in the wild locus, so we are done by the analytic case (\cref{wild_locus_analytic_curve}).

 In case $y_\eta$ does not have a center on~$C_Y$, the idea is to temporarily add the center to~$C_Y$.

 Being an element of~$Y$, the valuation~$y_\eta$ has a center $c_y \in C_Y^+$.
 Since~$y_\eta$ lies over~$\eta$, the center~$c_y$ is contained in the generic fiber.
 We denote by~$c_x$ the image of~$c_y$ in~$C_X^+$.
 It is the center of the image $x_\eta$ of $y_\eta$.

 As we are dealing with curves, the union $C_Y':=C_Y \cup \{c_y\}$ is again a smooth curve over~$k$ and an open subscheme of~$C_Y^+$ and similarly for $C_X':=C_X \cup \{c_x\}$.
 The induced morphism
 \[
  C_Y' \longrightarrow C_X'
 \]
 is étale when restricted to~$C_Y$ (obviously) but it is also étale at~$c_y$ because we assumed that $y_\eta$ is not contained in~$R_f$.
 We consider the morphism of adic spaces associated with the diagram of schemes
 \[
  \begin{tikzcd}
   C_Y'		\ar[r,open] \ar[d,"\text{étale}"'] & C_Y^+		\ar[d]  \\
   C_X'		\ar[r,open] \ar[d]					& C_X^+		\ar[d]  \\
   \Spec k	\ar[r,open]							& \Spec k^+.
  \end{tikzcd}
 \]
 We use the following notation for the corresponding morphism of adic spaces:
 \[
  f' \colon  Y':=\Spa(C_Y',C_Y^+) \longrightarrow X' := \Spa(C_X',C_X^+).
 \]
 Since $C_Y' \to C_X'$ is étale, also $f'$ is étale and it fits into a diagram
 \[
  \begin{tikzcd}
   Y	\ar[d,"f"']	\ar[r,open]	& Y'	\ar[d,"f'"]	\\
   X	\ar[r,open]				& X'.
  \end{tikzcd}
 \]
 The horizontal maps are open immersions.
 The complement of the image of~$Y$ in~$Y'$ consists of the points with support~$c_y$.
 Similarly, the complement of~$X$ in~$X'$ is given by the points with support~$c_x$.
 Note that the points with support~$c_y$ are either trivial or bounded.

 On~$Y'$ the horizontal specialization of~$y$ defined by~$y_\eta$ does exist.
 It is a point whose support is the closed point~$c_y$ of~$C_Y'$.
 This point has to be bounded (i.e. not trivial) as otherwise $x = x_\eta$, which is impossible in combination with our assumption that~$f$ is wild at~$x$ but strongly étale at~$x_\eta$.
 By the analytic case (\cref{wild_locus_analytic_curve}), there is a divisorial point $y' \in Y'$ where~$f'$ is wild.
 The support of any divisorial point of~$Y'$ is the generic point of~$C_Y'$.
 Therefore,~$y'$ is already contained in~$Y$ and we have found the point we were looking for.
\end{proof}

\section{The wild locus in higher dimensions} \label{sect:wild_higher_dim}

Let $(k,k^+)$ be an affinoid field of finite rank.
In this section we study étale morphisms of smooth adic spaces over~$(k,k^+)$.
If the dimension of the spaces is greater than one, we do not have as much insight into their geometric structure as in the curve case.
We especially lack a higher dimensional analogue of intervals or segments.
Accordingly, we do not aim for a detailed study of continuity properties of the wild degree.
It will be sufficient for our puposes to know that the wild locus contains a divisorial point.

Our approach to tackle the higher dimensional setting consists in reducing to the curve case by constructing suitable fibrations into curves.
Accordingly, the construction of these fibrations will be the first step in this section.
We start with a preparatory lemma.

\begin{lemma} \label{transcendence_basis}
 Let $(K,K^+)/(k,k^+)$ be a separable extension of valued fields of transcendence degree~$n$ such that the specialization field extension has transcendence degree~$F$.
 Then there exist algebraically independent elements $t_1,\ldots,t_F\in K^+$ such that $K/k(t_1,\ldots,t_F)$ is separable and $K^\succ/k^\succ(t^\succ_1,\ldots,t^\succ_F)$ is algebraic (where~$t^\succ_i$ denotes the image of~$t_i$ in the specialization field~$K^\succ$).
\end{lemma}

\begin{proof}
 Choose algebraically independent elements
 \[
  t_1,\ldots,t_F \in K^+
 \]
 such that the induced extension on specialization fields
 \[
  k^\succ(t^\succ_1,\ldots,t^\succ_F) \hookrightarrow K^\succ
 \]
 is algebraic.
 This is possible precisely by the definition of~$F$.

 We are almost done, but we still have to resolve the obstacle that $K/k(t_1,\ldots,t_F)$ might be inseparable. We will inductively replace the elements $t_i$ by elements $t'_i=t_i+x_i$ such that $K/k(t'_1,\ldots, t'_F)$ is separable and $|x_i|<1$, forcing that $k^\succ(t'^\succ_1,\ldots,t'^\succ_F)=k^\succ(t^\succ_1,\ldots,t^\succ_F)$ stays unchanged. By induction, we can assume that $L=k(t_1,\ldots,t_{i-1})$ is such that $K/L$ is separable and we should only modify $t_i$. Assume that $K/L(t_i)$ is inseparable, as otherwise we just keep $t_i$ unchanged. Choose any $x\in K$ which is transcendental over $L$ and such that $K/L(x)$ is separable. Replacing $x$ by $x^{-1}$ if needed we can assume that $|x|\le 1$ and then multiplying it by a non-zero $\pi\in k$ we can even assume that $|x|<1$. By our choice $dx\neq 0$ and $dt_i=0$ in $\Omega^1_{K/L}$. Therefore $t'_i=t_i+x$ satisfies $dt'_i\neq 0$, and this means that $K/L(t'_i)$ is separable.
\end{proof}

\begin{proposition} \label{factorization_curves}
 Let $(k,k^+)$ be an affinoid field of finite rank and $X \to \Spa(k,k^+)$ smooth of dimension~$n>0$.
 Let~$x$ be a point of $X$ with $F_x < n$.
 Then there is a strongly étale neighborhood~$U$ of~$x$ in~$X$ with trivial completed residue field extension over~$\widehat{k(x)}$, a smooth adic space $Y/(k,k^+)$, and a smooth morphism $g\colon U \to Y$ of relative dimension~$1$ over $(k,k^+)$ such that $F_{g(x)} = F_x$.
 If $(k,k^+)$ is analytic, we can even achieve that~$U$ is an \emph{open} neighborhood.
\end{proposition}

\begin{proof}
 Set $F = F_x$.
 Using \cref{transcendence_basis} we find $t_1,\ldots,t_F$ in $k(x)^+$ such that $\widehat{k(x)}$ is the completion of a separable extension of $k(t_1,\ldots,t_F)$ and $k^\succ(t_1^\succ,\ldots,t_F^\succ)/k^\succ$ has transcendence degree~$F$.
 Since $F < n$, the topological transcendence degree of $k(x)/k(t_1,\ldots,t_F)$ is at least~$1$.
 On a neighborhood $U \subseteq X$ of~$x$, the extension $k(x)/k(t_1,\ldots,t_F)$ defines a smooth morphism
 \[
  f \colon U \longrightarrow \D_F =\Spa(k \langle T_1,\ldots,T_F \rangle,k^+ \langle T_1,\ldots,T_F \rangle),
 \]
 sending~$x$ to the Gauß point~$\eta$ (note that $k \langle T_1,\ldots,T_F \rangle = k[T_1,\ldots,T_F]$ in case~$k$ carries the discrete topology).
 Then $F=F_x=F_\eta$ and~$f$ is smooth of relative dimension $n-F \ge 1$.

 Shrinking~$U$ if necessary, we find a factorization of~$f$ into
 \[
  U \longrightarrow \D_n \longrightarrow \D_{n-1} \longrightarrow \D_F,
 \]
 where the first map $U \to \D_n$ is étale (see \cite[Corollary 1.6.10]{Hu96}).
 The image point $y \in \D_{n-1}$ of~$x$ satisfies $F_\eta \le F_y \le F_x$, whence $F_y = F_x$.
 Moreover, $U \to \D_{n-1}$ is a smooth relative curve and thus satisfies all the required properties.
 \end{proof}

 So far we have mainly discussed analytic adic spaces.
 Via analytification we will be able to treat the discretely ringed case as well.
 However, we need to do some extra work to treat unbounded points.
 The next lemma will be used for this purpose.

\begin{lemma} \label{wild_horizontal_generalization}
 Let $(k,k^+)$ be a discrete affinoid field and $f \colon Y \to X$ an étale morphism of smooth adic spaces over~$(k,k^+)$.
 Suppose that~$y \in Y$ is a point with non-generic support contained in the wild locus~$R_f$.
 Then there is a nontrivial horizontal generalization~$y'$ of~$y$ with the same specialization field contained in~$R_f$.
\end{lemma}

\begin{proof}
 We consider the local ring homomorphism $\cO_{X,x} \to \cO_{Y,y}$ of regular local rings.
 Let $a_1,\ldots,a_n$ be a regular system of parameters in the maximal ideal of~$\cO_{X,x}$.
 By our assumption that the support of~$y$ is non-generic, we obtain that $n \ne 0$.
 Then~$a_1$ defines a discrete rank~$1$ valuation~$v_1$ on the fraction field of~$\cO_{X,x}$ such that its residue field equals the fraction field of $\cO_{X,x}/(a_1)$.
 More generally, $a_i$ defines such a valuation~$v_i$ on the fraction field of $\cO_{X,x}/(a_1,\ldots,a_{i-1})$ whose residue field equals the fraction field of $\cO_{X,x}/(a_1,\ldots,a_i)$.
 The valuations~$v_i$ can be composed to a valuation~$v$ of the fraction field of~$\cO_{X,x}$ with residue field equal to $k(x)$.
 Now in turn we can compose~$v$ with~$x$ to obtain a point~$x'$ of~$X$, which is a horizontal generalization of~$x$.
 Note that we do not have any continuity issues as we are dealing with discretely ringed adic spaces.

 Now we use that~$f$ is étale to conclude that $a_1,\ldots,a_n$ is also a regular system of parameters in~$\cO_{Y,y}$.
 The same construction as in~$\cO_{X,x}$ gives a valuation~$w$ of the fraction field of~$\cO_{Y,y}$ with residue field~$k(y)$.
 Composing it with~$y$, we obtain a point~$y'$ in~$Y'$, which is in fact the image of~$x'$.
 The extension $k(y')/k(x')$ is wild with respect to $y' = w \circ y$ because $k(y)/k(x)$ is wild with respect to~$y$.
\end{proof}

\begin{theorem} \label{R_f_divisorial_point}
 Let $f\colon Y \to X$ be an étale morphism of smooth adic spaces over an affinoid field $(k,k^+)$ of finite rank.
 Then the residually wild locus $R_f \subseteq Y$ is either empty or contains a divisorial point.
\end{theorem}

\begin{proof}
 Without loss of generality we may assume that~$X$ and~$Y$ are connected.
 In particular,~$X$ and~$Y$ are then pure of some dimension~$n$.
 Choose $y \in R_f$ and denote by~$x$ its image in~$X$.
 Suppose that~$y$, and hence~$x$, is not divisorial.
 So $F_x = F_y < n$.
 In the following we will find a new point $y' \in R_f$ over $x' \in X$ with $F_{x'} > F_x$.
 Repeating this process finitely many times will produce a divisorial point in $R_f$.

 We use \cref{factorization_curves} to find an open neighborhood~$U$ of~$x$, a smooth adic space~$Z$ over $(k,k^+)$, and a smooth morphism $g\colon  U \to Z$ over~$(k,k^+)$ of relative dimension~$1$ such that $F_x = F_{g(x)}$.
 So the specialization field extension $k(x)^\succ/k(g(x))^\succ$ is algebraic.
 Let~$C$ denote the fiber of~$g$ over $g(x)$.
 Then~$C$ is a smooth curve over $(k(g(x)),k(g(x))^+)$.

 Let us first assume that~$x$ is divisorial as a point of~$C$.
 The only way this can happen is when $k(g(x))$ is trivially valued and~$x$ is a discrete rank one valuation.
 Since we assumed that~$x$ is not divisorial as a point of~$X$, we can exclude the case where~$g(x)$ is the generic point of~$Z$.
 We are therefore in the position to apply \cref{wild_horizontal_generalization}.
 In this way we obtain a nontrivial horizontal generalization~$y''$ of~$y$ contained in~$R_f$ with $F_{y''} = F_y$.
 We denote by~$x''$ its image in~$X$.
 It is a nontrivial horizontal generalization of~$x$, hence contained in~$U$.

 Let us check that~$F_{g(x'')} = F_{x''}$.
 On the one hand, $g(x'')$ being a horizontal generalization of~$g(x)$, we know that $F_{g(x'')} \ge F_{g(x)} = F_x = F_{x''}$.
 On the other hand, $F_{g(x'')} \le F_{x''}$ as $x''$ maps to $g(x'')$.
 So in summary we get $F_{g(x'')} = F_{x''}$.

 The rank of the valuation corresponding to~$x''$ is at least~$2$.
 In particular, we cannot run into the same trap as for~$x$, or in other words~$x''$ cannot be divisorial in the fiber over $g(x'')$.
 Replacing~$x$ by~$x''$, we may assume that~$x$ is not divisorial as a point of~$C$ from the start.

 We consider the basechange of~$f$ to~$C$
 \[
  \begin{tikzcd}
   C_Y	\ar[r]	\ar[d,"f_C"]	& Y	\ar[d,"f"]	\\
   C		\ar[r]					& X.
  \end{tikzcd}
 \]
 Then~$y$ can be considered as a point of~$C_Y$ and~$x$ as a point of~$C$.
 In particular, $f_C$ is wild at $y$ as the completed residue field extension $\widehat{k(y)}/\widehat{k(x)}$ is wild.
 Using \cref{wild_locus_curves} we find a divisorial point $y'$ in the wild locus of~$f_C$.
 Then~$y'$ is also contained in the wild locus of~$f$.
 Moreover, the fact that $y'$ is divisorial as a point of~$C_Y$ and $y$ is not, translates to $F_{y'} > F_y$.
\end{proof}

\begin{corollary} \label{tame_divisorial_points}
 Let $f\colon Y \to X$ be an étale morphism of smooth adic spaces over an affinoid field $(k,k^+)$ of finite rank.
 Then~$f$ is tame if and only if it is tame at all divisorial points.
\end{corollary}

\begin{proof}
 This is a reformulation of the statement of \cref{R_f_divisorial_point}.
\end{proof}

\section{New insights on different notions of tameness} \label{sect:tameness_notions}

Let $X \to S$ be a morphism of finite type of noetherian schemes.
In this section we study finite étale morphisms $f\colon Y \to X$.
There are several ways of defining what it means for $f\colon Y \to X$ to be \emph{tame} with respect to~$S$.
In their article \cite{KeSch10} Moritz Kerz and Alexander Schmidt compare various approaches to the definition of tameness.
For some implications they need to assume resolution of singularities.
Our goal is to remove this assumption by resorting to the methods developed above.

We start by recalling the different concepts of tameness treated in \cite{KeSch10}.
We assume that~$S$ is integral, pure dimensional, separated, and excellent and consider the category $\Sch/S$ of separated schemes of finite type over~$S$.
The strongest tameness condition (which is, however, not included in \cite{KeSch10}) is the following.

\begin{definition}
 Let $f \colon Y \to X$ be a finite étale morphism in $\Sch/S$.
 We say that~$f$ is \emph{tame} if the induced morphism of adic spaces
 \[
  \Spa(Y,S) \longrightarrow \Spa(X,S)
 \]
 is tame.
\end{definition}

Tameness in this strong sense implies all other variants of tameness discussed in \cite{KeSch10}.
For the convenience of the reader we recall these in the following paragraphs.

We start with tameness for curves.
In some cases the dimension of schemes in $\Sch/S$ does not behave as expected.
For instance if $S$ is the spectrum of a discrete valuation ring with field of fractions $K$, $\Spec K$ is an open subscheme of~$S$ of dimension~$0$ but~$S$ has dimension~$1$.
But for our purposes we would rather like $\Spec K$ to have dimension~$1$ in this context.
Therefore, we use a modified definition of dimension that depends on the base scheme~$S$ (see \cite[\S~3]{KeSch10}).

For an integral scheme $X \in \Sch/S$ with scheme theoretic image $T \subseteq S$ (i.e. topologically~$T$ is the closure of the image of~$X$ {as~$X$ is quasi-compact), we define the $S$-\emph{dimension} of~$X$ by
\[
 \dim_S X := \trdeg ( K(X)/K(T)) + \dim T.
\]
Now we can define a \emph{curve} in $\Sch/S$ as an integral scheme $C \in \Sch/S$ with $\dim_S C = 1$.
By this definition $\Spec \Q_p$ is a curve over $\Z_p$.
If we start with a smooth morphism $X \to \Spec \Z_p$ of relative dimension~$1$, the special fiber is a curve over~$\Z_p$ but the generic fiber is not.
Moreover, the closure in~$X$ of every closed point of the generic fiber defines a curve over~$\Z_p$ (whether or not this closure contains points of the special fiber).

Every curve~$C$ has a unique compactification~$\bar{C}$ over~$S$ that is regular at the points of $\bar{C}\setminus C$.
For a finite étale morphism $C' \to C$ the compactifications of~$C$ and~$C'$ are compatible, i.e. we get an induced morphism
\[
 \bar{C}' \longrightarrow \bar{C}
\]
of compactifications over~$S$.
The finitely many points of $\bar{C}' \setminus C$ correspond to valuations of the function field $k(C')$.
We say that $C' \to C$ is \emph{tame} if these valuations are tame in the extension $k(C')/k(C)$ of function fields.

Tameness for curves lies the foundation for the definition of curve-tameness:

\begin{definition}
 A finite étale morphism $f\colon Y \to X$ in $\Sch/S$ is \emph{curve tame} if the base change of~$f$ to any curve in $\Sch/S$ is tame.
\end{definition}

\begin{lemma} \label{curve_tame_bc}
 Curve-tameness is stable under base change.
\end{lemma}

Curve tameness of $f\colon  Y \to X$ is equivalent to tameness of $\Spa(Y,S) \to \Spa(X,S)$ at points of the following type.
Let $y \in Y$ be a point such that its closure $\overline{\{y\}}$ is a curve.
In non-degenerate cases this means that~$y$ is of codimension~$1$.
Let~$v$ be an $S$-valuation of $k(y)$.
Then $(y,v)$ defines a point of $\Spa(Y,S)$ and curve-tameness requires tameness of $\Spa(Y,S) \to \Spa(X,S)$ at these points.
It is enough to check tameness at those points $(y,v)$ where~$v$ does not have center on~$Y$.
Otherwise $k(y)/k(x)$ (where $x = f(y)$) is automatically unramified with respect to~$v$ due to the étaleness of~$f$.

The concept that is closesest to the tameness of $\Spa(Y,S) \to \Spa(X,S)$ is the following

\begin{definition}
 A finite étale morphism $f\colon  Y \to X$ of normal connected schemes in $\Sch/S$ is \emph{valuation-tame} if for any $S$-valuation of $k(X)$ the function field extension $k(Y)/k(X)$ is tame.
\end{definition}

Equivalently, one can say that~$f$ is valuation-tame if $\Spa(Y,S) \to \Spa(X,S)$ is tame at all points whose support is the generic point of~$Y$.
Considerably weaker at first sight is the following concept.

\begin{definition}
 A finite étale morphism $f\colon  Y \to X$ of normal connected schemes in $\Sch/S$ is \emph{discrete-valuation-tame} if for any $S$-valuation of $k(X)$ which is a composition of $\dim_S X$ many discrete valuations of rank~$1$, the function field extension $k(Y)/k(X)$ is tame.
\end{definition}

For the definition of divisor-tameness we still assume that $X \in \Sch/S$ is normal.
Given a finite étale morphism $f \colon Y \to X$ we consider normal compactifications~$\bar{X}$ of~$X$ over~$S$.
Any prime divisor~$D$ of~$\bar{X}$ defines a valuation $v_D$ of the function field $k(X)$.

\begin{definition}
 A finite étale morphism $f\colon Y \to X$ of connected normal schemes in $\Sch/S$ is \emph{divisor-tame} if for any normal compactification~$\bar{X}$ over~$S$ and any divisor $D \subseteq \bar{X}$ the valuation~$v_D$ is tame in the function field extension $k(Y)/k(X)$.
\end{definition}

Because $f\colon Y \to X$ is étale, it suffices to check tameness for valuations~$v_D$ such that the divisor~$D$ is contained in $\bar{X}\setminus X$.

\begin{proposition} \label{description_divisorial_points_Spa}
 Let $(k,k^+)$ be a discrete affinoid field such that~$k^+$ is an excellent discrete valuation ring.
 We consider a diagram
 \[
  \begin{tikzcd}
   X	\ar[r]	\ar[d]	& \Spec k	\ar[d]	\\
   S	\ar[r]			& \Spec k^+
  \end{tikzcd}
 \]
 of finite type scheme morphisms such that~$S$ is integral, pure dimensional, and separated and~$X$ is normal and separated.
 Then the set of points $v_D$ for divisors~$D$ of a normal compactification of~$X$ over~$S$ are precisely the divisorial points of $\Spa(X,S)$ as an adic space over $(k,k^+)$.
\end{proposition}

\begin{proof}
 It is clear that the points~$v_D$ are divisorial.
 Suppose now that~$v$ is a divisorial point of $\Spa(X,S)$.
 The support of~$v$ is the generic point of~$X$ and since~$k^+$ has rank~$1$, the rank of~$v$ has to be~$1$ as well.
 If~$v$ is horizontally divisorial, the restriction of~$v$ to~$k$ is the trivial valuation and
 \[
  F_v = \dim_S X - 1.
 \]
 If~$v$ is vertically divisorial, the restriction of~$v$ to~$k$ has valuation ring~$k^+$ and
 \[
  F_v = \dim_S X.
 \]
 In both cases this means that Abhyankar's inequality is an equality.
 In \cite{KeSch10} they call these valuations \emph{geometric}.
 Therefore, by \cite[Proposition~3.2]{KeSch10}, there is a normal compactification~$\bar{X}$ of~$X$ over~$S$ such that the center $c_v \in \bar{X}$ of~$v$ has codimension~$1$ and satisfies
 \[
  k(c_v) = k(v)^\succ.
 \]
 Then the local ring $\cO_{\bar{X},c_v}$ is a discrete valuation ring dominated by the valuation ring $k(v)^+$ of $k(X)$.
 This implies
 \[
  k(v)^+ = \cO_{\bar{X},c_v}
 \]
 or in other words $v = v_D$ for the divisor $D := \overline{\{c_v\}}$.
\end{proof}

As a direct consequence we obtain the following result.

\begin{corollary}
 In the setting of \cref{description_divisorial_points_Spa}, let $f\colon  Y \to X$ be a finite étale morphism.
 Then~$f$ is divisor-tame if and only if $\Spa(Y,S) \to \Spa(X,S)$ is tame at all divisorial points.
\end{corollary}

Recall that a \emph{Parshin chain} on a scheme $X \in \Sch/S$ is a tuple $(P_0,P_1,\ldots,P_r)$ of points of~$X$ satisfying
\begin{itemize}
 \item $\overline{\{P_0\}} \subset \overline{\{P_1\}} \subset \ldots \subset \overline{\{P_r\}}$,
 \item $\dim_S \overline{\{P_i\}} = i$ for $i=0,\ldots,r$.
\end{itemize}
We call~$r$ the \emph{length} of the Parshin chain.

Suppose now that~$X$ is normal and connected.
For a discrete valuation~$v$ of $k(X)$ of rank~$r$ let us denote by~$v_i$ the unique generalization of rank~$i$
Then~$v$ \emph{dominates} the Parshin chain $(P_0,\ldots,P_r)$ if for each $i=0,\ldots,r$ the valuation ring $\cO_{v_i}$ dominates the local ring $\cO_{X,P_i}$.
There are only finitely many discrete valuation rings dominating a given Parshin chain and if the Parshin chain determines a regular sequence, there is a unique such valuation.

\begin{definition}
 Let $f \colon  Y \to X$ be a finite étale morphism of normal schemes in $\Sch/S$.
 \begin{enumerate}[(i)]
  \item	$f$ is \emph{chain tame} if there exists a normal compactification~$\bar{X}$ of~$X$ over~$S$ such that for every Parshin chain on~$\bar{X}$ of length $\dim_S X$ every discrete valuation dominating it is tame in $k(Y)/k(X)$.
  \item	$f$ is \emph{weakly chain tame} if there exists a normal compactification~$\bar{X}$ of~$X$ over~$S$ such that for every Parshin chain $(P_0,\ldots,P_r)$ on~$\bar{X}$ of length $r = \dim_S X$ such that $P_1,\ldots,P_r \in X$ every discrete valuation dominating it is tame in $k(Y)/k(X)$.
 \end{enumerate}
\end{definition}

The main theorem of \cite{KeSch10} comparing the above defined notions of tameness is Theorem~4.4.
It asserts (in slightly simplified form) the following.

\begin{theorem}[\cite{KeSch10}, Theorem~4.4] \label{comparison_tameness_KeSch}
 Let $f \colon  Y \to X$ be a finite étale morphism of regular schemes in $\Sch/S$.
 The following are equivalent:
 \begin{enumerate}[(i)]
  \item $f$ is curve-tame,
  \item $f$ is divisor-tame,
  \item $f$ is discrete-valuation tame,
  \item $f$ is chain-tame,
  \item $f$ is weakly chain-tame,
 \end{enumerate}
 In case resolution of singularities holds, these are equivalent to
 \begin{enumerate}[(i)]
  \setcounter{enumi}{5}
  \item $f$ is valuation-tame.
 \end{enumerate}
 If~$X$ is regular and admits a regular compactification~$\bar{X}$ over~$S$ such that $\bar{X} \setminus X$ is the support of an snc divisor, all of the above are equivalent to
 \begin{enumerate}[(i)]
  \setcounter{enumi}{6}
  \item $k(Y)/k(X)$ is tame at the valuations $v_D$ corresponding to the components of the boundary divisor.
 \end{enumerate}
\end{theorem}

The most difficult part is to show the implication from (i) to (ii).
Here the authors reduce to the case of surfaces and then do a thorough case by case study after resolving the singularities of the surface.
Our contribution to this topic is to remove the assumption on resolution of singularities when comparing (i)-(v) with~(vi).

\begin{theorem} \label{divisor_tame_implies_tame}
 Suppose~$S$ is of finite type over a Dedekind scheme~$T$.
 Let $f \colon  Y \to X$ be a finite étale morphism of regular schemes in $\Sch/S$.
 If~$f$ is divisor-tame, it is tame.
\end{theorem}

\begin{proof}
 Assume that~$f$ is divisor-tame, hence curve-tame by \cref{comparison_tameness_KeSch}.
 We want to show that $\Spa(Y,S) \to \Spa(X,S)$ is tame at a given point~$(y,v)$ of $\Spa(Y,S)$.
 The point~$y$ is mapped to a point of~$T$, whose inclusion we denote by $\Spec k \hookrightarrow T$.
 Let~$k^+$ be the valuation ring of the valuation~$v$ restricted to~$k$.
 By our assumption that~$T$ is a Dedekind scheme, $\Spec k^+ \to T$ is the localization of~$T$ at some point.
 We obtain an inclusion $\Spa(k,k^+) \to \Spa(T,T)$ such that~$(y,v)$ lies over the closed point of $\Spa(k,k^+)$.

 We now consider the base change of
 \[
  \begin{tikzcd}
   \Spa(Y,S)	\ar[rr]	\ar[dr]	&			& \Spa(X,S)	\ar[dl]	\\
								& \Spa(T,T)
  \end{tikzcd}
 \]
 along $\Spa(k,k^+) \to \Spa(T,T)$.
 We arrive at a morphism
 \[
  \begin{tikzcd}
   \Spa(Y',S')	\ar[rr]	\ar[dr]	&				& \Spa(X',S')	\ar[dl]	\\
								& \Spa(k,k^+)
  \end{tikzcd}
 \]
 with
 \begin{align*}
  Y'	& = Y \times_T \Spec k,	\\
  X'	& = X \times_T \Spec k				\\
  S'	& = S \times_T \Spec k^+.
 \end{align*}
 Since curve-tameness is stable under base change, the morphism $f' \colon  Y' \to X'$ is curve-tame as well.

 We are almost in the position to apply \cref{tame_divisorial_points}.
 The only issue we have is that we cannot guarantee that $\Spa(X',S')$ is smooth over $(k,k^+)$.
 In order to put ourselves in a better position, we first replace $Y' \to X'$ by its base change along $\overline{\{f(y)\}} \to X'$ (where $\overline{\{f(y)\}}$ carries the reduced scheme structure).
 In this way we achieve that~$y$ is a generic point of~$Y'$.

 Since $X'$ is of finite type over~$k$, there is a finite purely inseparable extension $k'/k$ such that $(X'_{k'})_\reduced$ is geometrically reduced.
 By \cref{purely_inseparable_wild} we can replace~$k$ by~$k'$.
 Moreover, we may replace~$X'$ by its underlying reduced subscheme.
 Then~$X'$ is geometrically reduced and thus generically smooth.
 Let~$U'$ be a smooth dense open subscheme of~$X'$ with preimage~$V'$ in~$Y'$.
 Then the restriction of~$f'$ to~$U'$ defines a finite étale morphism of smooth rigid analytic spaces over $(k,k^+)$:
 \[
  f'|_{U'} \colon  \Spa(V',S') \longrightarrow \Spa(U',S').
 \]
 Moreover, $V'$ is still curve-tame over~$U'$, hence divisor tame.
 It thus follows from \cref{tame_divisorial_points} that $f'|_{U'}$ is tame.
 Since $y$ is a generic point of~$Y'$, it is contained in~$V'$.
 So in particular, $f'$ is tame at $(y,v)$.
 This is equivalent to~$f$ being tame at $(y,v)$.
\end{proof}

We can now finally deduce tameness from curve-tameness.
When discussing divisor-tameness we had to work with normal schemes but curve-tameness is also well defined for non-normal schemes.
Accordingly, we do not require normality in the following statement.

\begin{corollary} \label{curve_tame_implies_tame}
 Suppose~$S$ is of finite type over a Dedekind scheme~$T$.
 Let $f \colon  Y \to X$ be a finite étale morphism in $\Sch/S$.
 If~$f$ is curve-tame, it is tame.
\end{corollary}

\begin{proof}
 Suppose that~$f$ is curve-tame.
 Let $(y,v) \in \Spa(Y,S)$ be a point with image $(x,w) \in \Spa(X,S)$.
 Then~$f$ is tame at $(y,v)$ if and only if the induced morphism
 \[
  f_y \colon \overline{\{y\}} \longrightarrow \overline{\{x\}}
 \]
 of topological closures is tame at $(y,v)$.
 Both~$\overline{\{y\}}$ and~$\overline{\{x\}}$ are integral schemes.
 Moreover, $\overline{\{y\}}$ is a connected component of the base change of~$f$ to~$\overline{\{x\}}$.
 In particular, $f_y$ is finite étale and curve tame by \cref{curve_tame_bc}.
 Base changing further to a non-empty open subset of~$\overline{\{x\}}$, we may assume that $\overline{\{x\}}$ and~$\overline{\{y\}}$ are regular.
 By \cref{comparison_tameness_KeSch}, $f_y$ is divisor-tame.
 Finally, we can apply \cref{divisor_tame_implies_tame} to conclude that~$f_y$ and thus~$f$ is tame at $(y,v)$.
\end{proof}

We would like to remark that \cref{divisor_tame_implies_tame} was already known before in the case where~$S$ is the spectrum of a field.
It is due to Alexander Schmidt (see \cite{quasi-purity} and uses alterations.

\section{Numerical tameness}

Numerical tameness is a concept that differs in nature from the tameness conditions treated in \cref{sect:tameness_notions} in the sense that it is not valuation theoretic.
Moreover, it depends on a chosen compactification.
The notion of numerical tameness arose when studying the Euler characteristic for coherent sheaves (see \cite{CE92})
It is a natural tameness condition when working with Galois modules and local systems.
Let us explain the setup.

We keep the base scheme~$S$ from \cref{sect:tameness_notions}, so~$S$ is integral, pure-dimensional, separated, and excellent.
In $\Sch/S$ we consider a finite étale Galois morphism $f \colon  Y \to X$ of normal connected schemes.
We choose a normal compactification~$\bar{X}$ of~$X$ over~$S$ and define~$\bar{Y}$ to be the normalization of~$\bar{X}$ in~$Y$.
To sum up, we have a Cartesian digram of schemes
\[
 \begin{tikzcd}
  Y	\ar[r,open]	\ar[d,"fet"',"\text{Galois}"]	& \bar{Y}	\ar[d,"\bar{f}"]	\\
  X	\ar[r,open]									& \bar{X}.
 \end{tikzcd}
\]
For a point $y \in \bar{Y}$ with image $x \in \bar{X}$ we consider the decomposition group $G_y(\bar{Y}/\bar{X})$ and the inertia group $I_y(\bar{Y}/\bar{X})$ in the Galois group $G = G(Y/X)$.
They are defined as follows.
\begin{align*}
 G_y(\bar{Y}/\bar{X}) &= \{g \in G \mid gy = y\},	\\
 I_y(\bar{Y}/\bar{X}) &= \ker(G_y(\bar{Y}/\bar{X}) \to \Aut(k(y)/k(x))).
\end{align*}
We say that~$\bar{f}$ is \emph{numerically tame} at a point $y \in \bar{Y}$ if the order of the inertia group $I_y(\bar{Y}/\bar{X})$ is prime to the residue characteristic of~$y$.
We say that $f$ is \emph{numerically tame(ly ramified) along} $\bar{X} \setminus X$ if~$\bar{f}$ is numerically tame at every point $y \in \bar{Y}$.
If we start with a finite étale morphism $f \colon Y \to X$ that is not necessarily Galois, we call it numerically tame if it can be dominated by a numerically tame Galois cover.

Numerical tameness of~$f$ along $\bar{X} \setminus X$ implies valuation tameness and hence all the other tameness conditions treated in \cref{sect:tameness_notions} (see \cite[, Theorem~5.3]{KeSch10}).
The key observation is that the inertia group $T_y(\bar{Y}/\bar{X})$ at a point $y \in \bar{Y}$ contains the inertia groups associated to all valuations of~$K(Y)$ centered on~$y$.
The converse is not always true.
There are tame covers that are not numerically tame for a given compactification (see \cite[Example~2 in the appendix]{KeSch10}).
However, if~$\bar{X}$ is smooth and $\bar{X} \setminus X$ is an snc divisor, numerical tameness does imply tameness:

\begin{theorem}[\cite{KeSch10}, Theorem~5.4]
 Let $\bar{X} \in \Sch/S$ be regular and proper and $X \subseteq \bar{X}$ a dense open subscheme such that $\bar{X} \setminus X$ is an snc divisor.
 If an étale covering $f \colon Y \to X$ is tamely ramified along $\bar{X} \setminus X$ (meaning that $K(Y)/K(X)$ is tamely ramified at the valuations corresponding to the irreducible components of $\bar{X} \setminus X$), then~$f$ is numerically tame along $\bar{X} \setminus X$.
\end{theorem}

If we start with a regular scheme~$X$ in $\Sch/S$ and assume that resolutions of singularities exist over~$S$, we can find a regular compactification~$\bar{X}$ of~$X$ whose boundary $\bar{X} \setminus X$ is an snc divisor.
For this compactification any finite étale morphism $f \colon Y \to X$ is numerically tame along $\bar{X} \setminus X$.
However, in positive and mixed characteristic, resolution of singularities is not known in higher dimensions and we cannot expect to find a compactification that works for all étale covers.
What we \emph{can} do is to construct a compactification that depends on the étale cover $Y \to X$.
This will be achieved by a compactness argument on the limit over all compactifications of~$X$.
Indeed, the limit over all compactifications is a reasonable locally ringed space that has been studied in \cite{Tem11} and carries the name \emph{relative Riemann Zariski space}:
\[
 \RZ_X(S) = \lim_{i \in I} \bar{X}_i,
\]
where the directed set~$I$ parametrizes all compactifications $\bar{X} \to S$ of~$X$.
By \cite[Corollary~3.4.7]{Tem11}, the relative Riemann Zariski space~$\RZ_X(S)$ can be viewed as the subspace of $\Spa(X,S)$ consisting of all points of $\Spa(X,S)$ that do not admit a nontrivial horizontal specialization.
The relative Riemann Zariski space $\RZ_X(X)$ is a quasi-compact spectral space.
If $x \in \RZ_X(S)$ is a point with image $x_i \in \bar{X}_i$ for $i \in I$, then
\[
 \cO_{\Spa(X,S),x}^+ = \colim_i \cO_{\bar{X}_i,x_i}.
\]

\begin{lemma} \label{numerically_tame_open}
 Let $f\colon Y \to X$ be a finite étale Galois morphism in $\Sch/S$ of connected normal schemes with normal compactifications~$\bar{X}$ and~$\bar{Y}$ such that~$\bar{Y}$ is the normalization of~$\bar{X}$ in~$Y$.
 \[
  \begin{tikzcd}
   Y	\ar[r,open]	\ar[d,"f"']	& \bar{Y}	\ar[d,"\bar{f}"]	\\
   X	\ar[r,open]				& \bar{X}
  \end{tikzcd}
 \]
 Then the set of all points $y \in \bar{Y}$ where~$\bar{f}$ is numerically tame is open.
\end{lemma}

\begin{proof}
 We denote by~$G$ the Galois group of the cover $f \colon  Y \to X$.
 Suppose that~$\bar{f}$ is numerically tame at a given point $y \in \bar{Y}$ of residue characteristic~$p$.
 The inertia subgroup~$I_y$ at~$y$ defines an intermediate cover:
 \[
  \begin{tikzcd}
   \bar{Y}	\ar[dd,bend left,"f"]	\ar[d,"I_y"']	\\
   \bar{Y}_1	\ar[d]	\\
   \bar{X}
  \end{tikzcd}
 \]
 Let~$y_1$ be the image of~$y$ in~$\bar{Y}_1$.
 Since $\bar{Y}_1 \to \bar{X}$ is étale at~$y_1$, there is an open neighborhood $U_1 \subseteq \bar{Y}_1$ of~$y_1$ such that $U_1 \to \bar{X}$ is étale.
 We can choose~$U_1$ small enough to ensure that the residue characteristic of any point in~$U_1$ is prime to $\lvert I_y \rvert$.
 Its preimage $U := U_1 \times_{\bar{Y}_1} \bar{Y}$ is an open neighborhood of~$y$ and the restriction of~$\bar{f}$ to~$U$ factors as
 \[
  U	\longrightarrow U_1	\longrightarrow X.
 \]
 We know that $U \times_{\bar{X}} X \to U_1 \times_{\bar{X}} X$ is finite Galois with Galois group~$I_y$, which is prime to all the residue characteristics of~$U$, so the first morphism is numerically tame.
 The second morphism is even étale, so~$\bar{f}$ is numerically tame in the open neighborhood~$U$ of~$y$.
\end{proof}

\begin{theorem} \label{compnumerical}
 Let $f \colon Y \to X$ be a finite étale cover of connected normal schemes in $\Sch/S$ such that $\Spa(Y,S) \to \Spa(X,S)$ is tame.
 Then there is a compactification~$\bar{X}$ of~$X$ over~$S$ such that~$f$ is numerically tame along $\bar{X} \setminus X$.
\end{theorem}

\begin{proof}
  It is easy to reduce to the case where~$f$ is Galois with Galois group~$G$.
  Let~$I$ be a directed set parametrizing the compactifications $\bar{X}_i$ of~$X$ for $i \in I$.
  For every compactification~$\bar{X}_i$ of~$X$ over~$S$ we denote by~$\bar{Y}_i$ the compactification of~$Y$ obtained by taking the normalization of~$\bar{X}_i$ in~$Y$.
  The subset $Z_i \subseteq \bar{Y}_i$ where $\bar{Y}_i \to \bar{X}_i$ is \emph{not} numerically tame is closed by \cref{numerically_tame_open}.
  We now consider a transition map $\varphi_{ij} \colon \bar{Y}_i \to \bar{Y}_j$.
  Since numerical tameness is preserved by base change (\cite[Lemma~5.2]{KeSch10}), the transition map~$\varphi_{ij}$ maps~$Z_i$ into~$Z_j$.
  We obtain a projective system $(Z_i)_{i \in I}$, whose limit~$Z$ is a closed subset of $\RZ_Y(S) = \lim_{i \in I} \bar{Y}_i$.

  Let us now study the ramification behaviour of a point $y \in \RZ_Y(S)$ under the morphism $\RZ_Y(S) \to \RZ_X(S)$.
  We denote by $y^{\ad}$ the image of~$y$ in the adic space $\Spa(Y,S)$ and by $\underline{y} \in Y$ the support of~$y^{\ad}$.
  Moreover, we use the same notation $x^{\ad} \in \Spa(X,S)$ and $\underline{x} \in X$ for the image~$x$ of~$y$ in~$\RZ_X(S)$.
  The residue field extension $k(y^{\ad})/k(x^{\ad})$ coincides with $k(\underline{y})/k(\underline{x})$ and is a Galois extension of valued fields with Galois group naturally isomorphic to
  \[
   G_{\underline{y}} = \{g \in G \mid g\underline{y} = \underline{y}\} \subseteq G.
  \]
  Contained in~$G_{\underline{y}}$ we have the decomposition group and the inertia group at~$y$:
  \begin{align*}
   G_y	&= \{g \in G \mid gy = y\},	\\
   I_y	&= \ker\left(G_y \longrightarrow \Aut (k(y)^\succ/k(x)^\succ)\right).
  \end{align*}
  The étale cover $\Spa(Y,S) \to \Spa(X,S)$ ramifies at~$y$ if and only if~$I_y$ is nontrivial and the ramification is wild if and only if $p=\ch k(y^\succ)$ divides the order of~$I_y$.

  For every $i \in I$ we denote by $y_i \in \bar{Y}_i$ and $x_i \in \bar{X}$ the centers of the valuations~$y$ and~$x$.
  So~$y$ can be identified with the projective system $(y_i)_{i \in I}$ and likewise for~$x$.
  Accordingly, the decomposition group $G_{y_i}$ is contained in~$G_y$ and
  \[
   G_y = \bigcap_{i \in I} G_{y_i}.
  \]
  The stalks of the structure sheaves of $\Spa(Y,S)$ at~$y^{\ad}$ and their residue fields are given by
  \begin{align*}
   \cO_{\Spa(Y,S),y^{\ad}}= \cO_{Y,\underline{y}},		\qquad				& k(y^{\ad}) = k(\underline{y})	\\
   \cO^+_{\Spa(Y,S),y^{\ad}} = \colim_{i \in I} \cO_{\bar{Y}_i,y_i}, \qquad & k(y^{\ad})^\succ = \colim_{i \in I} k(y_i).
  \end{align*}
  From this and the definition of the respective inertia groups we deduce the following relation.
  \[
   I_y = \bigcap_{i \in I} I_{y_i}
  \]
  Consequently, $Z = \lim_{i \in I} Z_i$ identifies with the wild locus of $\RZ_Y(S) \to \RZ_X(S)$, which is empty by hypothesis.

  We would now like to invoke a compactness argument to conclude that one of the~$Z_i$'s has to be empty.
  However, this does not work in the underlying topology of~$\bar{Y}_i$ and~$\RZ_Y(S)$ as the subspaces~$Z_i$ are only quasi-compact and not necessarily compact.
  To make this work we need to pass to the respective constructible topologies where~$Z_i$ are compact.
  Now we can conclude that~$Z_i$ has to be empty for some $i \in I$ and for this $i$ the morphism $f \colon Y \to X$ is numerically tame with respect to $\bar{X}_i \setminus X$.
\end{proof}

Combining \cref{compnumerical} with \cref{curve_tame_implies_tame} we obtain an answer to Hélène Esnault's question:

\begin{corollary} \label{compcurvenumerical}
 Suppose that~$S$ is integral, pure dimensional, separated, and of finite type over an excellent Dedekind scheme.
 For a regular scheme~$X$ of finite type over~$S$ and a curve tame cover $f \colon Y \to X$, we can find a compactification~$\overline{X}$ of~$X$ over~$S$ such that~$f$ is numerically tame with respect to~$\overline{X}$.
\end{corollary}

\bibliographystyle{meinStil}
\bibliography{citations}

\end{document}